\documentclass[a4paper,10pt,twoside]{amsart}
\usepackage{hyperref}
\usepackage{amsfonts}
\usepackage{amssymb}
\usepackage{amsmath}
\usepackage[cmtip,all]{xy}
\usepackage{lscape}
\usepackage[usenames]{color}

\def\white{\color{white}}

\newtheorem{proposition}{Proposition}[section]
 \newtheorem{lemma}[proposition]{Lemma}
 \newtheorem{theorem}[proposition]{Theorem}
 \newtheorem{corollary}[proposition]{Corollary}
\theoremstyle{definition}
 \newtheorem{definition}[proposition]{Definition}
 \newtheorem*{definition*}{Definition}
 
 \newtheorem*{example*}{Example}

\numberwithin{equation}{section}

\def\id{\mathsf{id}}
\def\RRU{\mathsf{RRU}}
\def\LRU{\mathsf{LRU}}
\def\RLU{\mathsf{RLU}}
\def\LLU{\mathsf{LLU}}
\def\RRC{\mathsf{RRC}}
\def\LRC{\mathsf{LRC}}
\def\RLC{\mathsf{RLC}}
\def\LLC{\mathsf{LLC}}
\def\bc{\begin{center}}
\def\ec{\end{center}}

\def\ox{\otimes}

\begin{document}

\title{Weak bimonoids in duoidal categories}

\author{Yuanyuan Chen}
\address{Department of Mathematics, Nanjing Agricultural University, Nanjing
210095, P.R. China.}

\author{Gabriella B\"ohm}
\address{Wigner Research Centre for Physics, Budapest,
H-1525 Budapest 114, P.O.B.\ 49, Hungary}
\email{bohm.gabriella@wigner.mta.hu}

\thanks{Y.Y. Chen participated in this research as a fellow of the Hungarian
Scholarship Board. Her work was supported also by the Innovative
Project of Jiangsu province for Graduate Cultivation.}

\begin{abstract}
{\em Weak bimonoids} in duoidal categories are introduced. They provide a
common generalization of bimonoids in duoidal categories and of weak bimonoids
in braided monoidal categories. Under the assumption that idempotent morphisms
in the base category split, they are shown to induce weak bimonads (in four
symmetric ways). As a consequence, they have four separable Frobenius base
(co)monoids, two in each of the underlying monoidal categories.
Hopf modules over weak bimonoids are defined by weakly lifting the
induced comonad to the Eilenberg-Moore category of the induced monad.
Making appropriate assumptions on the duoidal category in question, the
fundamental theorem of Hopf modules is proven which says that the category of
modules over one of the base monoids is equivalent to the category of Hopf
modules if and only if a Galois-type comonad morphism is an isomorphism. 
\end{abstract}
\date{2013 June}
\maketitle

\section*{Introduction}

Weak bialgebras were introduced first in the symmetric monoidal category of
(finite dimensional) vector spaces \cite{BSz,Nill:WBA,BNSz:WHAI}. They were
generalized to braided monoidal categories (with split idempotents) in
\cite{C. Pastro,AlAlatal}. The characteristic feature of weak bialgebras is
the behavior of the category of their (co)modules. Similarly to usual,
non-weak bialgebras, these categories are monoidal. However, as a
`weak' feature, the monoidal structure is different from that in the base
category. That is to say, the forgetful functor is no longer strict monoidal
as in the case of non-weak bialgebras, but it possesses a so-called `separable
Frobenius' monoidal structure \cite{Szlach}. This means that it is both
monoidal and opmonoidal but the morphisms which are responsible for these
structures, are not mutually inverses of each other. The binary part of the
monoidal structure only provides a left inverse of the binary
part of the opmonoidal structure and some compatibility conditions ---
reminiscent to those between the multiplication and the comultiplication of a
Frobenius algebra --- hold. This property of the forgetful functor provides
the basis of a generalization of weak bialgebras beyond braided monoidal base
categories; to so-called `weak bimonads' in \cite{G. Bohm2011}. 

Weak bimonads are monads on a monoidal category whose idempotent morphisms
split, equipped with the additional structures that are equivalent to their
Eilenberg-Moore category being monoidal with a forgetful functor possessing a
separable Frobenius monoidal structure. The separable Frobenius monoidal
forgetful functor takes the monoidal unit of the Eilenberg-Moore category to a
separable Frobenius monoid in the base category \cite{McCStr} which is
regarded as the `base monoid' of the weak bimonad in question. The monoidal
structure of the Eilenberg-Moore category is given in fact by the module
tensor product over the base monoid.

Braided monoidal categories were generalized to so-called `duoidal categories'
in \cite{M. Aguiar} (where they were termed `2-monoidal categories'). These
are categories carrying two, possibly different monoidal structures. The
monoidal structures are required to be compatible in the sense that the
functors and natural transformations defining the first monoidal structure,
are opmonoidal with respect to the second monoidal structure. Equivalently,
the functors and natural transformations defining the second monoidal
structure, are monoidal with respect to the first monoidal structure (for more
details see Section \ref{sec:duo_cat}). Also bimonoids in duoidal
categories were defined in \cite{M. Aguiar}. These are objects which are
monoids with respect to the first monoidal product $\circ$ and comonoids with
respect to the second monoidal product $\bullet$. The compatibility axioms are 
formulated in terms of the coherence morphisms between the monoidal
structures. These bimonoids generalize bialgebras in braided monoidal
categories in such a way that their categories of modules (respectively,
comodules) are still monoidal via the monoidal product $\bullet$
(respectively, $\circ$) lifted from the base category. In other words, they
induce `bimonads' (termed `Hopf monads' in \cite{I. Moerdijk}) with
respect to $\bullet$ (respectively, `bicomonads' with respect to $\circ$), see
\cite{T. Booker}.

The first aim of this paper is to find a common generalization of weak
bimonoids in braided monoidal categories and of bimonoids in duoidal
categories.
So we consider an object in a duoidal category, which is a monoid with respect
to the first monoidal structure and a comonoid with respect to the
second monoidal structure. We look for compatibility conditions between them
which imply the expected behavior of the (co)module categories: their
monoidality via the (co)module tensor product over some separable Frobenius
base monoid.
The proposed axioms of what we call a {\em weak bimonoid} are presented in
Section \ref{sec:axioms} where we also study their behavior under the various
duality transformations in a duoidal category. As a main result, a weak
bimonoid in a duoidal category in which idempotent morphisms split, is shown
to induce four weak bi(co)monads (two on each of the underlying monoidal
categories).
Corresponding to the four induced weak bi(co)monads, there are four associated
`base objects' (two in each of the underlying monoidal categories), all of
them carrying the structure of separable Frobenius monoid. Their properties
are investigated in Section \ref{sec:modcat} and their relations to each other
are studied in Section \ref{sec:base}.

Associated to any bimonoid in a duoidal category, there is a mixed
distributive law (in the sense of \cite{J. Beck}, also known as an `entwining
structure', see \cite{BrzMaj}) between the induced monad and comonad. Its
mixed modules are known as `Hopf modules' \cite{G. Bohm2012}.
In order to describe Hopf modules over weak bialgebras (even over a field),
however, mixed distributive laws had to be generalized to `weak mixed
distributive laws' in \cite{CaeDeGr} (where they were termed `weak
entwining structures'). In Section \ref{sec:rel_Hopf} we construct a weak mixed
distributive law between the monad and the comonad induced by a weak bimonoid
in a duoidal category. Hopf modules are defined as its mixed modules. Applying
the theory of weak liftings in \cite{Bohm:WTM}, we also construct a
(comparison) functor from the category of modules over one of the
base monoids to the category of Hopf modules. In Section \ref{sec:fthm} we
prove the fundamental theorem of Hopf modules: Under appropriate
assumptions on the duoidal category in question, we show that this
comparison functor is an equivalence if and only if a canonical Galois-type
comonad morphism is an isomorphism. Recall that these equivalent
properties distinguish weak Hopf algebras between weak bialgebras 
\cite[Section 36.16]{BrzWis}. 
In contrast to \cite{G. Bohm2012} --- where in our study of non-weak bimonoids
we used directly Beck's theory \cite{J. Beck} to construct the inverse of the
comparison functor --- here we take a shorter route. The proof in Section
\ref{sec:fthm} is based on a recent result due to Mesablishvili and Wisbauer
in \cite{B. Mesablishvili} about the properties of a functor occurring in the
factorization of some separable left adjoint functor on a category whose
idempotent morphisms split. 

\section{Preliminaries}\label{sec:prelims}

\subsection{Duoidal categories} \label{sec:duo_cat}
In this section we recall from \cite{M. Aguiar} some information about
so-called duoidal (also known as 2-monoidal) categories. The occurring
monoidal structures are not assumed to be strict. However, for brevity, we
omit explicitly denoting the associator and the unitors. Throughout, the
composition of any morphisms $\varphi$ and $\psi$ is denoted by
$\varphi.\psi$. 

\begin{definition} \cite[Definition 6.1]{M. Aguiar}
A {\it duoidal category} is a category $\mathcal{M}$ equipped with two
monoidal products $\circ$ and $\bullet$ with respective units $I$ and
$J$, along with morphisms
$$
\xymatrix{I\ar[r]^-{\delta} &I\bullet I,} \ \
\xymatrix{J\circ J\ar[r]^-{\varpi}&J,} \ \
\xymatrix{I\ar[r]^-\tau &J,}
$$
and, for all objects $A,B,C,D$ in $\mathcal M$, a morphism
$$
\zeta_{A,B,C,D}:(A\bullet B)\circ(C\bullet D)\rightarrow (A\circ
C)\bullet(B\circ D),
$$
called the {\it interchange law}, which is natural in all of the four
occurring objects. These morphisms are required to obey the axioms
below.

\textbf{Compatibility of the units.} The monoidal units $I$ and
$J$ are compatible in the sense that
\begin{center}
$(J,\varpi,\tau)$ is a monoid in $(\mathcal{M},\circ,I)$ and
$(I,\delta,\tau)$ is a comonoid in $(\mathcal{M},\bullet,J)$.
\end{center}

\textbf{Associativity.} For all objects $A,B,C,D,E,F$ in $\mathcal
M$, the following diagrams commute.
\begin{equation} \label{eq1.1}
\xymatrix{
(A\bullet B)\circ(C\bullet D)\circ(E\bullet F)
\ar[d]_{\zeta\circ (E\bullet F)}
\ar[r]^{(A\bullet B)\circ\zeta}& 
(A\bullet B)\circ((C\circ E)\bullet(D\circ F))\ar[d]^{\zeta}\\ 
((A\circ C)\bullet(B\circ D))\circ(E\bullet F)\ar[r]_-{\zeta}& 
(A\circ C\circ E)\bullet(B\circ D\circ F) 
\\
(A\bullet B\bullet C)\circ(D\bullet E\bullet F)
\ar[d]_{\zeta}
\ar[r]^{\zeta}& 
(A\circ D)\bullet((B\bullet C)\circ(E\bullet F))
\ar[d]^{(A\circ D)\bullet\zeta}\\ 
((A\bullet B)\circ(D\bullet E))\bullet(C\circ F)
\ar[r]_-{\zeta\bullet (C\circ F)} & 
(A\circ D)\bullet(B\circ E)\bullet (C \circ F) 
}
\end{equation}

\textbf{Unitality.} For all objects $A,B$ in $\mathcal M$, the
following diagrams commute.
\begin{equation}\label{eq1.2}
\xymatrix{
A\bullet B \ar[r]^-{\delta\circ (A\bullet B)} \ar@{=}[rd]
\ar[d]_-{(A\bullet B)\circ \delta}& 
(I\bullet I)\circ(A\bullet B)\ar[d]^-{\zeta} & 
(J\circ J)\bullet(A\circ B) \ar[r]^-{\varpi\bullet (A\circ B)} & 
A\circ B \\ 
(A\bullet B)\circ (I\bullet I)\ar[r]_-\zeta& 
A \bullet B & 
A \circ B \ar[u]^-{\zeta}\ar@{=}[ru] \ar[r]_-\zeta& 
(A\circ B)\bullet (J\circ J) \ar[u]_-{(A\circ B)\bullet \varpi} 
}
\end{equation}
\end{definition}

The simplest examples of duoidal categories are braided monoidal
categories. In this case, both monoidal products coincide and the interchange
law is induced by the braiding, see \cite[Section 6.3]{M. Aguiar}.

In any duoidal category, for any objects $A$ and $B$, diagrams of the type
\begin{equation}\label{eq:coh}
\xymatrix@C=40pt{
(A\bullet I)\circ B \ar[d]_-\zeta \ar[r]^-{(A\bullet \tau)\circ B}&
A\circ B\ar[d]_-\zeta\ar@{=}@/^1.3pc/[rd]
\ar@{}[rd]|-{\eqref{eq1.2}}\\
A\circ B \ar[r]^-{(A\circ B)\bullet (\tau\circ J)}
\ar@{=}@/_1.2pc/[rr]&
(A\circ B)\bullet (J\circ J)\ar[r]^-{(A\circ B)\bullet\varpi}&
A\circ B}
\end{equation}
commute, see \cite[Proposition 6.8]{M. Aguiar}.

\subsection{Weak bimonads}\label{sec:wbm}
The modules over a weak bialgebra (say, over a field) constitute a
monoidal category. However, the monoidal product is different from the tensor
product of vector spaces. This behavior of the category of modules was the
basis of the generalization of weak bialgebras to weak bimonads in
\cite{G. Bohm2011}.

{\em Weak bimonads} on a monoidal category (with monoidal product $\ox$
and monoidal unit $K$) were defined in \cite[Definition 1.3]{G. Bohm2011} as
monads equipped with a monoidal structure of their Eilenberg-Moore category of
modules (or algebras) and a separable Frobenius monoidal structure (in the
sense of \cite{Szlach}) of the forgetful functor to the base
category. Whenever idempotent morphisms in the base category split, this
definition turns out to be equivalent to an opmonoidal structure 
$(T_2:T(-\ox -)\to T(-)\ox T(-),T_0:TK\to K)$ subject to five compatibility
axioms in \cite[Theorem 1.5]{G. Bohm2011} with the monad structure
$(\mu:T^2\to T,\eta:\id \to T)$.

Recall from \cite[proof of Proposition 1.11]{G. Bohm2011} that if $T$ is
a weak bimonad on a monoidal category in which idempotent morphisms split,
then for any $T$-modules $(F,\varphi)$ and $(G,\gamma)$ there is an idempotent
morphism 
\begin{equation}\label{eq:chi_gen}
\xymatrix{
\chi_{F,G}: F\ox G\ar[r]^-{\eta_{F\otimes G}}
& T(F\otimes G)\ar[r]^-{T_2}
& TF\otimes TG\ar[r]^-{\varphi\otimes\gamma}
& F\otimes G}
\end{equation}
which is natural in $(F,\varphi)$ and $(G,\gamma)$. The monoidal product of
the $T$-modules $(F,\varphi)$ and $(G,\gamma)$ is the object through which
this idempotent morphism splits. Also the monoidal unit $R$ in the category of
$T$-modules is obtained by splitting an idempotent morphism
\begin{equation}\label{sqcap1}
\sqcap:\xymatrix@C=30pt{
TK \ar[r]^-{\eta TK}
& T^2K \ar[r]^-{T_2}
& TK\otimes T^2K\ar[r]^-{TK\otimes\mu_K}
& TK\otimes TK\ar[r]^-{TK \otimes T_0}
& TK}
\end{equation}
as $\sqcap=\xymatrix@C=20pt{TK \ar@{->>}[r]|-{\,\pi\,} & R\,
\ar@{>->}[r]|-{\,\iota\,}& TK.}$ 
By unitality of the monad $T$ and counitality of its opmonoidal
structure,
\begin{equation}\label{eq:T0_Pi}
\sqcap.\eta K=\eta K
\qquad \textrm{and}\qquad
T_0.\sqcap=T_0.
\end{equation}
By \cite[eq. (1.10)]{G. Bohm2011}, the morphism $\sqcap$ in \eqref{sqcap1}
renders commutative 
\begin{equation}\label{eq:wbm_1.10}
\xymatrix{
T^2K\ar[r]^-{T\sqcap}\ar[d]_-{\mu K}&
T^2K\ar[r]^-{\mu K}&
TK\ar[d]^-\sqcap\\
TK\ar[rr]_-\sqcap&&
TK.}
\end{equation}
The $T$-action on $R$ is $\pi.\mu_K.T\iota$ (so that $\pi$ is a morphism of
$T$-modules by \eqref{eq:wbm_1.10}).
Moreover, $R$ carries the structure of a separable Frobenius monoid in the
base category (with multiplication and comultiplication obtained by the
splitting
$\chi_{R,R}=\xymatrix@C=25pt{R\ox R \ar@{->>}[r]|-{\, \mu^R\,}&
R\, \ar@{>->}[r]|-{\,\Delta^R\,}&R\ox R}$ of \eqref{eq:chi_gen},
see \cite[eqs. (2.1)-(2.2)]{G. Bohm2011}); and the monoidal product of
$T$-modules turns out to be a module tensor product over $R$.
By the second equality in \eqref{eq:T0_Pi} and by commutativity of 
$$
\xymatrix{
TK\ar[rrr]^-\pi\ar[dd]_-{T_2}\ar[rd]^-{T_2}&&&
R\ar[dd]^-{\Delta_R}\\
&TK\ox TK\ar[r]^-{\pi\ox \pi}\ar[ld]_-{\chi}&
R\ox R\ar[rd]^-{\chi}\ar[ru]^-{\mu_R}\\
TK\ox TK\ar[rrr]_-{\pi \ox \pi}&&&
R\ox R,}
$$
$\pi:TK \to R$ is a morphism of comonoids. (The region on the left commutes by
the weak bimonad axiom \cite[eq. (1.7)]{G. Bohm2011} and unitality of the
monad $T$; and the region at the top commutes by the naturality of $\chi$,
counitality of the opmonoidal functor $T$ and the second equality in 
\eqref{eq:T0_Pi} and $T_2=\chi.T_2$ again, taking into account the explicit
expression of $\mu_R$ in \cite[eq. (2.1)]{G. Bohm2011}.) 

Whenever $T$ is a weak bimonad on a monoidal category (whose idempotent
morphisms split), it is a weak bimonad (with the same monad, and opmonoidal
structures) on the opposite monoidal category as well. Corresponding to this
latter weak bimonad on the opposite monoidal category, there is a symmetric
counterpart
$$
\overline\sqcap:\xymatrix@C=30pt{
TK \ar[r]^-{\eta TK}
& T^2K \ar[r]^-{T_2}
& T^2K\otimes TK\ar[r]^-{\mu K\otimes TK}
& TK\otimes TK\ar[r]^-{T_0 \otimes TK}
& TK}
$$
of $\sqcap$ in \eqref{sqcap1}. It is shown in \cite[page 12]{G. Bohm2011} that
$\sqcap. \overline \sqcap = \sqcap$. Symmetrically, $\overline \sqcap. \sqcap
= \overline \sqcap$. Hence taking the splittings 
$$
\sqcap =\xymatrix@C=15pt{
TK \ar@{->>}[r]^-{\pi}&
R\, \ar@{>->}[r]^-{\iota}&
TK}
\quad \textrm{and}\quad
\overline\sqcap=\xymatrix@C=15pt{
TK \ar@{->>}[r]^-{\overline \pi}&
\overline R\, \ar@{>->}[r]^-{\overline \iota}&
TK}
$$
of these idempotent morphisms, we obtain mutually inverse isomorphisms
\begin{equation}\label{eq:R_Rbar_iso}
\xymatrix@C=15pt{
R \ar[r]^-{\iota}&
TK \ar[r]^-{\overline \pi}&
\overline R}
\quad \textrm{and}\quad
\xymatrix@C=15pt{
\overline R \ar[r]^-{\overline \iota}&
TK \ar[r]^-{\pi}&
R.}
\end{equation}
Together with the natural transformation \eqref{eq:chi_gen}, they render
commutative 
$$
\xymatrix@R=10pt{
R \otimes R 
\ar[rr]^-{\chi_{R,R}}
\ar[rd]_-{\mu^R}
\ar[ddd]_-{\iota \otimes \iota}&&
R \otimes R
\ar[ddd]^-{\iota \otimes \iota}\\
&R
\ar[ru]_-{\Delta^R}
\ar[d]^-{\Delta^R}\\
&
R \otimes R 
\ar[d]^-{\iota \otimes \iota}\\
TK \otimes TK
\ar[ddd]_-{\overline\pi\otimes \overline\pi}&
TK \otimes TK
\ar[d]^-{\overline\pi\otimes \overline\pi}&
TK \otimes TK
\ar[ddd]^-{\overline\pi\otimes \overline\pi}\\
&
\overline R\otimes \overline R
\ar[d]^-{\mu^{\overline R}}\\
&\overline R
\ar[rd]^-{\Delta^{\overline R}}\\
\overline R\otimes \overline R
\ar[ru]^-{\mu^{\overline R}}
\ar[rr]_-{\chi_{\overline R,\overline R}}&&
\overline R\otimes \overline R .}
$$
The vertical arrow $R \to \overline R$ at the middle obeys
\begin{eqnarray*}
\mu^{\overline R}.(\overline\pi\otimes \overline\pi).(\iota \otimes \iota).
\Delta^R.\pi &=&
\mu^{\overline R}.(\overline\pi\otimes \overline\pi).(\iota \otimes \iota).
(\pi \otimes \pi).\Delta^{TK}\\
&=&\mu^{\overline R}.(\overline\pi\otimes \overline\pi).\Delta^{TK}=
\mu^{\overline R}.\Delta^{\overline R}.\overline\pi=
\overline\pi=
\overline\pi.\iota.\pi.
\end{eqnarray*}
In the first and the third equalities we used that $\pi$ and $\overline \pi$
are comultiplicative. In the second and the last equalities we used that
$\overline \sqcap. \sqcap = \overline \sqcap$ hence
$\overline\pi.\iota.\pi=\overline\pi$. In the fourth equality we used that
$\mu^{\overline R}.\Delta^{\overline R}$ is the identity (that is, the
separability of $\overline R$). Since $\pi$ is an epimorphism, this
proves that the vertical arrow $R \to \overline R$ at the middle of the
diagram is equal to the isomorphism in \eqref{eq:R_Rbar_iso}, hence it is both
multiplicative and comultiplicative. It is also unital and counital by
\eqref{eq:T0_Pi} and the explicit expressions of the unit and the counit in
\cite[(2.1) and (2.2)]{G. Bohm2011}. This proves that \eqref{eq:R_Rbar_iso} 
is an isomorphism of monoids and comonoids. 

The forgetful functor, from the Eilenberg-Moore category of $T$-modules to the
base category, factorizes through the category of $R$-bimodules \cite[page
14]{G. Bohm2011}. That is, any $T$-module carries canonical (commuting) left
and right $R$-actions and any morphism of $T$-modules is compatible with
them. For example, \eqref{eq:chi_gen} is a morphism of $T$-modules (hence so
are its splitting mono- and epimorphisms), thus it is a morphism of
$R$-bimodules. 

In the diagram
\begin{equation}\label{eq:theta_forms}
\xymatrix@C=30pt{
T^2K \ar[r]^-{TT_2}\ar@{=}[rd]
& T(TK\otimes TK)\ar[r]^-{T_2}\ar[d]^-{T(T_0 \ox TK)}
& T^2K\otimes T^2K\ar[r]^-{T^2K\ox \mu_K}\ar[d]^-{TT_0\ox T^2K}
& T^2K\otimes TK\ar[d]^-{T^2K\ox T_0}\\
& T^2K\ar[r]_-{T_2}
& TK\otimes T^2K\ar[d]^-{TK\ox \mu_K}
& T^2K\ar[d]^-{TT_0} \\
&& TK\otimes TK\ar[r]_-{TK\ox T_0}
& TK,}
\end{equation}
the triangle on the left commutes by counitality of the opmonoidal structure,
and both regions on the right commute by naturality of $T_2$ and functoriality
of the monoidal product. This allows to write $\sqcap$ in the alternative form 
\begin{equation}\label{resqcap}
\xymatrix@C=25pt{
TK \ar[r]^-{\raisebox{5pt}{${}_{\eta TK}$}}
& T^2K \ar[r]^-{\raisebox{5pt}{${}_{TT_2}$}}
& T(TK\otimes TK)\ar[r]^-{\raisebox{5pt}{${}_{T_2}$}}
& T^2K\otimes T^2K\ar[r]^-{\raisebox{5pt}{${}_{T^2K\ox \mu_K}$}}
& T^2K\otimes TK\ar[r]^-{\raisebox{5pt}{${}_{TT_0\ox T_0}$}}
& TK.}
\end{equation}

\section{The weak bimonoid axioms}\label{sec:axioms}

In this section we introduce the central notion of the paper: weak bimonoid
in a duoidal category. It is shown to provide a common generalization of
bimonoids in duoidal categories \cite[Definition 6.25]{M. Aguiar}, and of weak
bimonoids in braided monoidal categories \cite{C. Pastro}. If idempotent
morphisms in the base category split, any weak bimonoid is proven to induce
four bi(co)monads.

\begin{definition}\label{Def:weak bimonoid}
A {\it right weak bimonoid} in a duoidal category $\mathcal{M}$
is an object $A$ equipped with a monoid structure $(A\circ A \stackrel \mu \to
A,I\stackrel \eta\to A)$ in $(\mathcal{M},\circ)$ and a comonoid structure
$(A \stackrel \Delta\to A \bullet
A,A\stackrel \varepsilon\to J)$ in $(\mathcal{M},\bullet)$,
subject to the five compatibility axioms listed below.

In the same way as for a bimonoid, the multiplication is required to be
comultiplicative, equivalently, the comultiplication is required to be
multiplicative:
\begin{equation}\label{eq:WB}
\xymatrix {
(A\bullet A)\circ(A\bullet A)\ar[rr]^-{\zeta}
&& (A\circ A)\bullet (A\circ A)\ar[d]^-{\mu\bullet\mu}\\
A\circ A\ar[u]^-{\Delta\circ\Delta}\ar[r]_-{\mu} & A\ar[r]_-{\Delta} &
A\bullet A
}
\end{equation}
Comultiplicativity of the unit is replaced by two weaker conditions
$$
\xymatrix @R=15pt@C=11pt{
((I\bullet A)\circ (I\bullet A\bullet A))\bullet A
\ar[d]_-{\zeta\bullet A}\ar@{}[rrrrdd]|-{(\RRU)}&&
((I\bullet A)\circ (I\bullet I))\bullet A
\ar[ll]_-{((I\bullet A)\circ (I\bullet\Delta.\eta))\bullet A}&&
I\bullet A\bullet A
\ar[ll]_-{((I\bullet A)\circ \delta)\bullet A}\\
I\bullet A\bullet(A\circ A)\bullet A
\ar[d]_-{ I\bullet A\bullet\mu\bullet A}&&&&
I\bullet I \ar[u]_-{I\bullet\Delta.\eta}\\
I\bullet A\bullet A\bullet A&
I\bullet A\bullet A \ar[l]^-{I\bullet A\bullet\Delta }&
I\bullet A\ar[l]^-{I\bullet \Delta}&
I\bullet I\ar[l]^-{I\bullet\eta}&
I \ar[u]_-{\delta}\ar[l]^-{\delta}\\
A\bullet((A\bullet I)\circ (A\bullet A\bullet I))
\ar[d]_-{A\bullet\zeta}\ar@{}[rrrrdd]|-{(\LRU)}&&
A\bullet((A\bullet I)\circ (I\bullet I))
\ar[ll]_-{A\bullet((A\bullet I)\circ (\Delta.\eta\bullet I))}&&
A\bullet A\bullet I
\ar[ll]_-{A\bullet((A\bullet I)\circ\delta)}\\
A\bullet (A\circ A)\bullet A\bullet I
\ar[d]_-{A\bullet\mu\bullet A\bullet I} &&&&
I\bullet I \ar[u]_-{\Delta.\eta\bullet I}\\
A\bullet A\bullet A\bullet I&
A\bullet A\bullet I \ar[l]^-{\Delta\bullet A\bullet I}&
A\bullet I\ar[l]^-{\Delta\bullet I}&
I\bullet I\ar[l]^-{\eta\bullet I}&
I. \ar[u]_-{\delta}\ar[l]^-{\delta}}
$$
Multiplicativity of the counit is replaced by two weaker conditions
$$
\xymatrix @R=15pt@C=10pt{
((J\circ A\circ A)\bullet(J\circ A))\circ A
\ar@{}[rrrrdd]|-{(\LRC)}
\ar[rr]^-{((J\circ\varepsilon.\mu)\bullet(J\circ A))\circ A}
&& ((J\circ J)\bullet(J\circ A))\circ A
\ar[rr]^-{(\varpi\bullet(J\circ A))\circ A}
&& J\circ A\circ A\ar[d]^-{J\circ\varepsilon.\mu} \\
J\circ A\circ(A\bullet A)\circ A \ar[u]^{\zeta\circ A}
&&&& J\circ J\ar[d]^-{\varpi} \\
J\circ A\circ A\circ A \ar[u]^-{J\circ A\circ\Delta\circ A}
\ar[r]_-{J\circ A\circ\mu}
& J\circ A\circ A\ar[r]_-{J\circ \mu}
& J\circ A\ar[r]_-{J\circ \varepsilon}
& J\circ J \ar[r]_-{\varpi}
& J\\
((J\circ A)\bullet(J\circ A\circ A))\circ A
\ar@{}[rrrrdd]|-{(\LLC)}
\ar[rr]^-{((J\circ A)\bullet(J\circ\varepsilon.\mu))\circ A}
&& ((J\circ A)\bullet(J\circ J))\circ A
\ar[rr]^-{((J\circ A)\bullet\varpi)\circ A}
&& J\circ A\circ A\ar[d]^-{J\circ\varepsilon.\mu} \\
J\circ A\circ(A\bullet A)\circ A \ar[u]^{\zeta\circ A}
&&&& J\circ J\ar[d]^-{\varpi} \\
J\circ A\circ A\circ A \ar[u]^-{J\circ A\circ\Delta\circ A}
\ar[r]_-{J\circ A\circ\mu}
& J\circ A\circ A\ar[r]_-{J\circ \mu}
& J\circ A\ar[r]_-{J\circ \varepsilon}
& J\circ J \ar[r]_-{\varpi}
& J.}
$$
\end{definition}
A duoidal structure can be twisted in three different ways, see \cite[Section
6.1.2]{M. Aguiar}.
One can replace any of the monoidal products by the opposite one; and as
a third possibility, one can interchange the roles of both monoidal
products and change simultaneously the composition of morphisms to the
opposite one. That is, associated to any duoidal category $(\mathcal M, \circ,
\bullet)$, there are seven more duoidal categories $(\mathcal M,
\circ^{\mathrm{op}}, \bullet)$, $(\mathcal M, \circ, \bullet^{\mathrm{op}})$,
$(\mathcal M, \circ^{\mathrm{op}}, \bullet^{\mathrm{op}})$, 
$(\mathcal M^{\mathrm{op}}, \bullet,\circ)$, $(\mathcal M^{\mathrm{op}},
\bullet,\circ^{\mathrm{op}})$, $(\mathcal M^{\mathrm{op}},
\bullet^{\mathrm{op}},\circ)$, $(\mathcal M^{\mathrm{op}},
\bullet^{\mathrm{op}},\circ^{\mathrm{op}})$. 

Alternatively, changing the passive and active points of view, to any diagram
in $\mathcal M$ one can apply the transformations $\circ\leftrightarrow
\circ^{\mathrm{op}}$, $\bullet \leftrightarrow \bullet^{\mathrm{op}}$ or
$(\circ\leftrightarrow \bullet, \mathcal M \leftrightarrow \mathcal
M^{\mathrm{op}})$ --- which will be denoted by $\circ$, $\bullet$ and $\ast$,
respectively --- and any of their composites. Clearly, the diagrams defining a
monoid in $(\mathcal M,\circ)$ are invariant under the transformation $\circ$
or $\bullet$; and they are taken by $\ast$ to the diagrams defining a comonoid
in $(\mathcal M,\bullet)$. Symmetrically, the diagrams defining a comonoid in
$(\mathcal M,\bullet)$ are invariant under the transformation $\circ$ or
$\bullet$; and they are taken by $\ast$ to the diagrams defining a monoid in
$(\mathcal M,\circ)$. What is more, any of the transformations $\circ$,
$\bullet$ and $\ast$ leaves invariant the set of diagrams defining a bimonoid
in $(\mathcal M,\circ,\bullet)$. 

Our next task is to study the behavior of the weak bimonoid axioms
$(\RRU), (\LRU), (\LLC), (\LRC)$ in Definition \ref{Def:weak bimonoid} under
these transformations. The set of these diagrams is closed under the action of
$\bullet$ (which connects $(\RRU)$ with $(\LRU)$, and $(\LLC)$ with
$(\LRC)$). It is not closed, however, under the action of the other two
transformations. We obtain a set of diagrams which are closed under all
transformations $\circ$, $\bullet$ and $\ast$, if we add the images
under $\circ$. We denote the images of $(\RRU), (\LRU), (\LLC), (\LRC)$ under
$\circ$ by $(\RLU), (\LLU), (\RLC), (\RRC)$, respectively. Explicitly, they
read as
$$
\xymatrix @R=15pt@C=10pt{
((I\bullet A\bullet A)\circ (I\bullet A))\bullet A
\ar[d]_-{\zeta\bullet A}\ar@{}[rrrrdd]|-{(\RLU)}&&
((I\bullet I)\circ (I\bullet A))\bullet A
\ar[ll]_-{((I\bullet\Delta.\eta)\circ (I\bullet A))\bullet A}&&
I\bullet A\bullet A
\ar[ll]_-{(\delta\circ(I\bullet A))\bullet A}\\
I\bullet A\bullet(A\circ A)\bullet A
\ar[d]_-{ I\bullet A\bullet\mu\bullet A}&&&&
I\bullet I \ar[u]_-{I\bullet\Delta.\eta}\\
I\bullet A\bullet A\bullet A &
I\bullet A\bullet A \ar[l]^-{I\bullet A\bullet\Delta }&
I\bullet A\ar[l]^-{I\bullet \Delta}&
I\bullet I\ar[l]^-{I\bullet\eta}&
I \ar[u]_-{\delta}\ar[l]^-{\delta}\\
A\bullet((A\bullet A\bullet I)\circ (A\bullet I))
\ar[d]_-{A\bullet\zeta} \ar@{}[rrrrdd]|-{(\LLU)}&&
A\bullet((I\bullet I)\circ (A\bullet I))
\ar[ll]_-{A\bullet((\Delta.\eta\bullet I)\circ (A\bullet I))}&&
A\bullet A\bullet I
\ar[ll]_-{A\bullet(\delta\circ(A\bullet I))}\\
A\bullet (A\circ A)\bullet A\bullet I
\ar[d]_-{A\bullet\mu\bullet A\bullet I} &&&&
I\bullet I \ar[u]_-{\Delta.\eta\bullet I}\\
A\bullet A\bullet A\bullet I&
A\bullet A\bullet I \ar[l]^-{\Delta\bullet A\bullet I}&
A\bullet I\ar[l]^-{\Delta\bullet I}&
I\bullet I\ar[l]^-{\eta\bullet I}&
I \ar[u]_-{\delta}\ar[l]^-{\delta}\\
A\circ((A\circ A\circ J)\bullet(A\circ J))
\ar@{}[rrrrdd]|-{(\RRC)}
\ar[rr]^-{A\circ((\varepsilon.\mu\circ J)\bullet(A\circ J))}
&& A\circ((J\circ J)\bullet(A\circ J))
\ar[rr]^-{A\circ (\varpi\bullet(A\circ J))}
&& A\circ A\circ J\ar[d]^-{\varepsilon.\mu\circ J} \\
A\circ(A\bullet A)\circ A\circ J \ar[u]^{A\circ\zeta}
&&&& J\circ J\ar[d]^-{\varpi} \\
A\circ A\circ A\circ J \ar[u]^-{A\circ\Delta\circ A\circ J}
\ar[r]_-{\mu\circ A\circ J}
& A\circ A\circ J\ar[r]_-{\mu\circ J}
& A\circ J\ar[r]_-{\varepsilon\circ J}
& J\circ J \ar[r]_-{\varpi}
& J\\
A\circ((A\circ J)\bullet(A\circ A\circ J))
\ar@{}[rrrrdd]|-{(\RLC)}
\ar[rr]^-{A\circ((A\circ J)\bullet(\varepsilon.\mu\circ J))}
&& A\circ((A\circ J)\bullet(J\circ J))
\ar[rr]^-{A\circ((A\circ J)\bullet\varpi)}
&& A\circ A\circ J\ar[d]^-{\varepsilon.\mu\circ J} \\
A\circ(A\bullet A)\circ A\circ J \ar[u]^{A\circ\zeta}
&&&& J\circ J\ar[d]^-{\varpi} \\
A\circ A\circ A\circ J \ar[u]^-{A\circ\Delta\circ A\circ J}
\ar[r]_-{\mu \circ A \circ J}
& A\circ A\circ J\ar[r]_-{\mu\circ J}
& A\circ J\ar[r]_-{\varepsilon\circ J}
& J\circ J \ar[r]_-{\varpi}
& J}
$$
We obtain the following picture of the actions on these diagrams.
$$
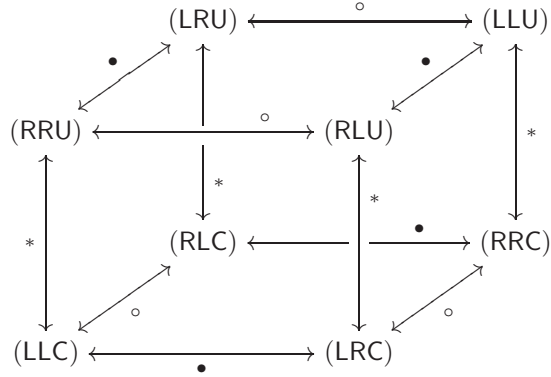
\xymatrix{
& (\LRU) \ar@{<->}[rr]^-{\circ} \ar@{<-}[d] 
&& (\LLU)\ar@{<->}[dd]^-{\ast} \\ 
(\RRU)\ar@{<->}[rr]^(.7){\circ}
\ar@{<->}[dd]_-{\ast}\ar@{<->}[ru]^-{\bullet} 
& \ar@{->}[d]^-{\ast} 
& (\RLU) \ar@{<->}[dd]^(.3){\ast}\ar@{<->}[ru]^-{\bullet}\\ 
& (\RLC) \ar@{<-}[r] 
& \ar@{->}[r]^-{\bullet} 
& (\RRC) \\ 
(\LLC)\ar@{<->}[rr]_-{\bullet}\ar@{<->}[ru]_-{\circ} 
&& (\LRC)\ar@{<->}[ru]_-{\circ} 
}
$$

In the next two propositions we present two large classes of examples of right
weak bimonoids in duoidal categories.

\begin{proposition}
Any bimonoid in a duoidal category satisfies the right weak bimonoid
axioms in Definition \ref{Def:weak bimonoid}. That is, right weak
bimonoid provides a generalization of bimonoid in a duoidal category.
\end{proposition}

\begin{proof}
If $A$ is a bimonoid, then the multiplication is comultiplicative by
assumption. Since $\Delta.\eta=(\eta\bullet \eta).\delta$, ($\RRU$) holds by
commutativity of
$$
\xymatrix{
I \bullet A \bullet A \ar[r]^-{\raisebox{5pt}{${}_{
((I \bullet A)\circ \delta^2)\bullet A}$}}&
((I \bullet A)\circ (I \bullet I \bullet I)) \bullet A
\ar[rr]^-{((I \bullet A)\circ (I \bullet \eta \bullet \eta)) \bullet A }&&
((I \bullet A)\circ (I \bullet A \bullet A)) \bullet A
\ar[d]^-{\zeta \bullet A}\\
I \bullet I \bullet I \ar[u]^-{I \bullet \eta \bullet \eta}
\ar[r]^-{\raisebox{5pt}{${}_{((I \bullet I)\circ \delta^2)\bullet I}$}}
\ar@/_1.5pc/[rr]_-{\delta \bullet I \bullet I}^-{(\ast)}&
((I \bullet I)\circ (I \bullet I \bullet I)) \bullet I
\ar[r]^-{\zeta \bullet I}&
I \bullet I \bullet I \bullet I
\ar[r]^-{I \bullet \eta \bullet (\eta \circ \eta) \bullet \eta}
\ar[rd]^-{I \bullet \eta \bullet \eta \bullet \eta}&
I \bullet A \bullet (A \circ A) \bullet A
\ar[d]^-{I \bullet A \bullet \mu \bullet A}\\
I\ar[rr]_-{\delta^3}\ar[u]^-{\delta^2}&&
I \bullet I \bullet I \bullet I
\ar[r]_-{I \bullet \eta \bullet \eta \bullet \eta}&
I \bullet A \bullet A \bullet A,}
$$
where, as usual, we denoted
$\delta^2=(I \bullet \delta).\delta=(\delta\bullet I).\delta$,
$\delta^3=(I \bullet \delta^2).\delta=(\delta^2 \bullet I).\delta$, and so
on. The top region commutes by functoriality of both monoidal products and
naturality of $\zeta$. The triangle at the bottom right commutes by the
unitality of $\mu$. The region marked by $(\ast)$ commutes by naturality
of $\zeta$ and a unitality axiom in \eqref{eq1.2}. The bottom region
commutes by the coassociativity of $\delta$. 

Since the bimonoid axioms are invariant under any of the transformations
$\circ$, $\bullet$ and $\ast$, all other axioms of right weak bimonoid in
Definition \ref{Def:weak bimonoid} hold by symmetry. 
\end{proof}

\begin{proposition}
Regard any braided monoidal category $(\mathcal M, \ox, I, \sigma)$ as a
duoidal category with coinciding monoidal structures $(\ox, I)$ and the
interchange law provided by the braiding $\sigma$. Then a weak bimonoid in the
sense of \cite{C. Pastro} is the same as a right weak bimonoid in
Definition \ref{Def:weak bimonoid}. In this way, right weak bimonoids in
duoidal categories provide a generalization also of weak bimonoids in braided
monoidal categories.
\end{proposition}

\begin{proof}
The multiplication is comultiplicative by definition, for any weak bimonoid in
the sense of \cite{C. Pastro} and also in the sense of
Definition \ref{Def:weak bimonoid}.
Axiom ($\RRU$) is equivalent to commutativity of the exterior of the following
diagram; and axiom (w.2) in \cite{C. Pastro} --- expressing weak
comultiplicativity of the unit --- is equivalent to commutativity of the
bottom region in
$$
\xymatrix{
A \ox A\ar[rr]^-{A\ox \eta \ox A}\ar[rrd]^-{\eta \ox A \ox A}&&
A \ox A \ox A \ar[rr]^-{A\ox \Delta \ox A}&&
A\ox A \ox A\ox A\ar[d]^-{\sigma_{A,A} \ox A \ox A}\\
A\ar[u]^-\Delta\ar[r]_-{\eta \ox A}&
A \ox A\ar[r]_-{A\ox \Delta}&
A \ox A\ox A\ar[r]_-{\Delta \ox A\ox A}
\ar[u]_-{\sigma_{A,A}^{-1}\ox A}&
A \ox A\ox A\ox A\ar[r]_-{A \ox \sigma_{A,A}^{-1}\ox A}
\ar[ur]^(.28){\sigma^{-1} _{A,A\ox A}\ox A}&
A\ox A\ox A\ox A\ar[d]^-{A \ox \mu \ox A}\\
I\ar[u]^-\eta\ar[rr]_-\eta&&
A\ar[rr]_-{\Delta^2}&&
A\ox A\ox A.}
$$
So they are equivalent by commutativity of all regions at the top: The
leftmost region commutes by functoriality of the monoidal product and the
remaining regions commute by coherence and naturality of the braiding.

The remaining three axioms are shown to be pairwise equivalent symmetrically.
\end{proof}

As a most important justification of Definition \ref{Def:weak bimonoid}, right
weak bimonoids induce weak bimonads (in the sense of \cite{G. Bohm2011}). The
proof of this fact starts with the following.

\begin{lemma}\label{lem:kappa}
For a right weak bimonoid $A$, and any objects $X$ and $Y$ in a duoidal
category $(\mathcal M, \circ, \bullet)$, consider the morphism
$$
\kappa_{X,Y}: \xymatrix@C=10pt{
((J\circ A)\bullet X)\circ (A\bullet Y)\ar[r]^-\zeta&
(J\circ A \circ A)\bullet (X\circ Y)
\ar[rr]^-{\raisebox{7pt}{${}_{
(J\circ \varepsilon.\mu)\bullet (X\circ Y) }$}}&&
(J\circ J)\bullet (X\circ Y) \ar[r]^-{\raisebox{7pt}{${}_{
\varpi \bullet (X\circ Y)}$}}&
X\circ Y.}
$$
The following assertions hold.
\begin{itemize}
\item[{(1)}] $\kappa$ is natural both in $X$ and $Y.$
\item[{(2)}] The following diagram commutes, for any objects $X,Y,Z$ in
 $\mathcal M$. 
$$
\xymatrix@C=45pt{
((J\circ A)\bullet X)\circ (A\bullet Y\bullet Z)\ar[r]^-\zeta
\ar[d]_-{\kappa_{X,Y\bullet Z}}&
(J\circ A \circ (A\bullet Y))\bullet (X\circ Z)
\ar[d]^-{\kappa_{J,Y}\bullet (X\circ Z)}\\
X\circ (Y\bullet Z)\ar[r]_-\zeta&
(J\circ Y)\bullet (X\circ Z).}
$$
\item[{(3)}] The following diagram commutes, for any objects $X,Y$ in
 $\mathcal M$. 
$$
\xymatrix@C=45pt{
((J\circ A)\bullet X)\circ (A\bullet Y)\ar[d]_-{\kappa_{X,Y}}
\ar[r]^-{\raisebox{6pt}{${}_{
((J\circ A\circ \Delta.\eta)\bullet X)\circ (A\bullet Y)}$}}&
((J\circ A\circ (A\bullet A))\bullet X)\circ (A\bullet Y)
\ar[d]^-{\raisebox{-7pt}{${}_{(\kappa_{J,A}\bullet X)\circ (A\bullet Y)}$}}\\
X\circ Y&
((J\circ A)\bullet X)\circ (A\bullet Y).
\ar[l]^-{\kappa_{X,Y}}}
$$
\end{itemize}
\end{lemma}

\begin{proof}
(1) is evident by naturality of $\zeta$ and functoriality of both
 monoidal products.

(2) follows easily by one of the associativity conditions in \eqref{eq1.1} and
 functoriality of $\bullet$.

(3) is proven in the Appendix, on page
\pageref{page_lem:kappa}. Commutativity of the region labelled by $(\LRC)$
in the diagram on page \pageref{page_lem:kappa} follows by precomposing both
equal paths in $(\LRC)$ with $\tau\circ A\circ \eta\circ A$. The undecorated
regions commute by naturality and the associativity of $\varpi$.
\end{proof}

\begin{theorem} \label{thm:wbm}
Any right weak bimonoid $A$ in a duoidal category
$(\mathcal{M},\circ,\bullet)$ in which idempotent morphisms split, induces a
weak bimonad $(-)\circ A$ on $(\mathcal{M},\bullet)$. 
\end{theorem}

\begin{proof}
The multiplication and the unit of the monad are induced
by the multiplication and the unit of the monoid $A$, respectively.
The binary part of the opmonoidal structure is given by
$$
\xymatrix{
(M\bullet M')\circ A \ar[rr]^-{(M\bullet M')\circ\Delta}
&& (M\bullet M')\circ(A\bullet A)\ar[r]^-{\zeta}
& (M\circ A)\bullet (M'\circ A),
}
$$
for all objects $M,M'$ in $\mathcal M$. The nullary part is provided by
$\xymatrix@C=15pt{
J\circ A \ar[r]^-{J\circ\varepsilon}
& J\circ J \ar[r]^-{\varpi}
& J. }$
This equips $(-)\circ A$ with the structure of an opmonoidal functor indeed:
coassociativity follows from the coassociativity of $\Delta$,
the associativity axioms \eqref{eq1.1} in a duoidal category and naturality of
$\zeta$; and counitality follows by the counitality of $\Delta$, the
unitality axioms \eqref{eq1.2} in a duoidal category and naturality of
$\zeta$ again. (The proof of this is identical to the non-weak case.)

We use the equivalent conditions in \cite[Theorem 1.5]{G. Bohm2011}
to show that $(-)\circ A$ is a weak bimonad.
First, \cite[eq. (1.7)]{G. Bohm2011} holds true by one of the
associativity axioms of a duoidal category in \eqref{eq1.1}, the
compatibility condition \eqref{eq:WB}, naturality of $\zeta$ and
functoriality of $\circ $ (this proof is identical to the non-weak
case). Verification of \cite[eq. (1.5)]{G. Bohm2011}, for any objects $X$, $Y$
and $Z$ of $\mathcal M$, can be found in the Appendix, on page
\pageref{page_1.5}. Commutativity of the region labelled by $(\RRU)$ in the
diagram on page \pageref{page_1.5}, follows by postcomposing both sides of
axiom $(\RRU)$ with $\tau \bullet A \bullet A \bullet A$. The undecorated
regions commute by naturality.
Finally, \cite[eq. (1.4)]{G. Bohm2011} is proven to hold, for any object
$X$ in $\mathcal M$, in the Appendix, on page \pageref{page_1.4}. The
proof makes use of Lemma \ref{lem:kappa}. Commutativity of the region labelled
by $(\LRU)$ in the diagram on page \pageref{page_1.4}, follows by
postcomposing both sides of $(\LRU)$ with $A \bullet A \bullet A\bullet
\tau$. The vertical arrow on the right-hand-side of the diagram on page
\pageref{page_1.4} is equal to the identity morphism by one of the unitality
axioms in \eqref{eq1.2}. 
The remaining diagrams in \cite[Theorem 1.5]{G. Bohm2011} can be obtained from
\cite[eqs. (1.4) and (1.5)]{G. Bohm2011} replacing $\bullet$ by its
opposite. Hence they hold true by symmetry.
\end{proof}

The transformations $\circ$, $\bullet$ and $\ast$ relate also the
functors induced by any object $A$, as depicted in the diagram
\begin{equation}\label{eq:induced_wbm}
\xymatrix@C=1pt{
\ar@(ul,dl)@{<->}[]_-{\bullet}&
(-)\circ A \ar@{<->}[rrrrrrrr]^-{\ast}\ar@{<->}[d]_-{\circ}&&&&&&&& 
(-)\bullet A&
\ar@(ur,dr)@{<->}[]^-{\circ}
\\
\ar@(dl,ul)@{<->}[]^-{\bullet}&
A\circ(-)\ar@{<->}[rrrrrrrr]^-{\ast}&&&&&&&& 
A\bullet(-).\ar@{<->}[u]_-{\bullet}&
\ar@(ur,dr)@{<->}[]^-{\circ}
}\end{equation}
Hence by symmetry considerations, from Theorem \ref{thm:wbm} we
obtain the following. (By a {\em weak bicomonad} below, we mean a weak
bimonad on the opposite category.)

\begin{corollary} \label{cor:wbms}
Let $\mathcal M$ be a duoidal category with monoidal products $\circ$ and
$\bullet$, such that idempotent morphisms in $\mathcal M$ split. Let $A$ be
an object of $\mathcal M$ which carries the structures of a monoid wrt $\circ$
and a comonoid wrt $\bullet$. Assume that \eqref{eq:WB} holds true.
\begin{itemize}
\item[{(1)}] If $(\RLU)$, $(\LLU)$, $(\RLC)$, and $(\RRC)$
hold, then $A\circ(-)$ is a weak bimonad on $(\mathcal M, \bullet)$.
\item[{(2)}] If $(\LLC)$, $(\RLC)$, $(\RRU)$ and $(\RLU)$
hold, then $(-)\bullet A$ is a weak bicomonad on $(\mathcal M,\circ)$.
\item[{(3)}] If $(\LRC)$, $(\RRC)$, $(\LRU)$ and $(\LLU)$
hold, then $A\bullet(-)$ is a weak bicomonad on $(\mathcal M,\circ)$.
\end{itemize}
\end{corollary}
This motivates the following.

\begin{definition}
A {\em left weak bimonoid} in a duoidal category $(\mathcal M,\circ, \bullet)$
is an object $A$ which carries the structures of a monoid wrt $\circ$ and a
comonoid wrt $\bullet$ such that \eqref{eq:WB} and the conditions $(\RLU)$,
$(\LLU)$, $(\RLC)$, and $(\RRC)$ hold. We say that $A$ is a {\em weak
bimonoid} if it is both a right weak bimonoid and a left weak bimonoid.
\end{definition}

In a braided monoidal category --- regarded as a duoidal category --- all
notions of right weak bimonoid, left weak bimonoid and weak bimonoid
coincide. Bimonoids in duoidal categories also provide examples
of both right weak bimonoids and left weak bimonoids (hence of weak bimonoids).

\section{The monoidal category of modules}\label{sec:modcat}

In light of Theorem \ref{thm:wbm}, we can apply all the information about weak
bimonads in Section \ref{sec:wbm} to the weak bimonad $(-)\circ A$ on
$(\mathcal M,\bullet)$, induced by a right weak bimonoid $A$ in a
duoidal category $(\mathcal M,\circ,\bullet)$ in which idempotent morphisms
split.

The monoidal product of any $(-)\circ A$-modules $(F,\varphi)$ and
$(G,\gamma)$ is given by splitting the idempotent morphism
\eqref{eq:chi_gen}, that is,
\begin{equation}\label{chi}
\xymatrix@C=15pt{
\chi^R_{F,G}:
F\bullet G\ar[rr]^-{(F\bullet G)\circ\Delta.\eta}
&& (F\bullet G)\circ (A\bullet A)\ar[r]^-{\zeta}
& (F\circ A)\bullet(G\circ A)\ar[r]^-{\varphi\bullet\gamma}
& F\bullet G.}
\end{equation}
The monoidal unit of the category of $(-)\circ A$-modules --- to
be denoted by $R_\circ$ --- is obtained by splitting the idempotent morphism
\eqref{sqcap1}, taking now the explicit form
\begin{equation}\label{eq:sqcap^R}
\xymatrix@C=8pt{
\sqcap^{R}_\circ:
J\circ A\ar[rr]^-{\raisebox{7pt}{${}_{
J\circ A\circ \Delta.\eta}$}}&&
J\circ A\circ (A\bullet A)\ar[r]^-{\raisebox{7pt}{${}_{\zeta}$}}&
(J\circ A)\bullet(J\circ A\circ A)
\ar[rr]^-{\raisebox{7pt}{${}_{
(J\circ A)\bullet(J\circ\varepsilon.\mu)}$}} &&
(J\circ A)\bullet(J\circ J)
\ar[rr]^-{\raisebox{7pt}{${}_{(J\circ A)\bullet\varpi}$}}&&
J\circ A.
}\end{equation}
By \eqref{resqcap}, the duoidal category axioms \eqref{eq1.1} and
\eqref{eq1.2}, by the counitality of $\Delta$, naturality of $\zeta$ and
functoriality of $\circ$, $\sqcap^{R}_\circ$ admits an
equal form
\begin{equation}\label{eq:sqcap_alt}
\xymatrix @R=8pt{
J\circ A
\ar[r]^-{J\circ A \circ \Delta.\eta}&
J\circ A\circ (A\bullet A)
\ar[r]^-{J\circ \zeta}&
J\circ ((J \circ A)\bullet(A\circ A))
\ar[r]^-{\raisebox{7pt}{${}_{J\circ ((J\circ A)\bullet \varepsilon.\mu)}$}}&
J\circ J\circ A\ar[r]^-{\varpi\circ A}&
J\circ A.}
\end{equation}
For the splittings of $\sqcap^{R}_\circ$ and $\chi^R_{F,G}$, the notations
$$\xymatrix {
J\circ A \ar[rr]^-{\sqcap^{R}_\circ}
\ar@{->>}[dr]_-{\pi^R_\circ}
&& J\circ A
&& F\bullet G \ar[rr]^-{\chi^R_{F,G}}\ar@{->>}[rd]_-{\pi^R_{F,G}}&&
F\bullet G
\\
& R_\circ \ar@{>->}[ru]_-{\iota^R_\circ} &
&&
& F \bullet_{R_\circ} G \ar@{>->}[ur]_-{\iota ^R_{F,G}}
}$$
will be used.
The retract $R_\circ$ is a separable Frobenius monoid in $(\mathcal
M,\bullet)$, with multiplication, unit, comultiplication and counit as below,
see \cite[eqs. (2.1)-(2.2)]{G. Bohm2011}.
\begin{equation}\label{eq:R_sepFrob}
\xymatrix @R=8pt @C=20pt{
\mu^R_\circ:& R_\circ\bullet R_\circ\ar[r]^-{\chi^R}
& R_\circ\bullet R_\circ\ar[r]^-{R_\circ\bullet\iota^R_\circ}
& R_\circ\bullet (J\circ A)\ar[r]^-{R_\circ\bullet (J\circ\varepsilon)}
& R_\circ\bullet (J\circ J)\ar[r]^-{R_\circ\bullet\varpi}
& R_\circ,
\\
\eta^R_\circ:& J\ar[r]^-{J\circ\eta}
& J\circ A\ar[r]^-{\pi^R_\circ}
& R_\circ,
\\
\Delta^R_\circ:& R_\circ\ar[r]^-{R_\circ\bullet(J\circ\eta)}
& R_\circ\bullet(J\circ A)\ar[r]^-{R_\circ\bullet\pi^R_\circ}
& R_\circ\bullet R_\circ\ar[r]^-{\chi^R}
& R_\circ\bullet R_\circ,
\\
\varepsilon^R_\circ:& R_\circ\ar[r]^-{\iota^R_\circ}
& J\circ A\ar[r]^-{J\circ\varepsilon}
& J\circ J\ar[r]^-{\varpi}
& J.
}
\end{equation}
By construction, $\pi^R_\circ:J\circ A \to R_\circ$ is a morphism of
comonoids, hence it induces (left and right) $R_\circ$-coactions on $J\circ
A$.

\begin{lemma}\label{lem:pi_colinear}
For any right weak bimonoid $A$ in a duoidal category $(\mathcal M, \circ,
\bullet)$ in which idempotent morphisms split, $\sqcap^R_\circ:J\circ A \to
J\circ A$ in \eqref{eq:sqcap^R} (and hence the epimorphism $\pi^R_\circ:J\circ
A \to R_\circ$ and the monomorphism $\iota^R_\circ:R_\circ\to J\circ A$ in its
splitting) are morphisms of right $R_\circ$-comodules.
\end{lemma}

\begin{proof}
The idempotent $\sqcap^R_\circ:J\circ A \to J\circ A$ is a morphism of right
$R_\circ$-comodules if and only if
\begin{equation}\label{eq:sqcap_colinear}
\xymatrix{
J\circ A \ar[r]^-{J\circ \Delta}\ar[d]_-{\sqcap^R_\circ}&
J\circ (A\bullet A)\ar[r]^-\zeta&
(J\circ A)\bullet(J\circ A)\ar[r]^-{(J\circ A)\bullet\sqcap^R_\circ}&
(J\circ A)\bullet(J\circ A)\ar[d]^-{\sqcap^R_\circ \bullet (J\circ A)}\\
J\circ A \ar[r]^-{J\circ \Delta}&
J\circ (A\bullet A)\ar[r]^-\zeta&
(J\circ A)\bullet(J\circ A)\ar[r]^-{(J\circ A)\bullet\sqcap^R_\circ}&
(J\circ A)\bullet(J\circ A)}
\end{equation}
commutes. Using expression \eqref{eq:sqcap^R} of $\sqcap^R_\circ$, the
coassociativity of $\Delta$, naturality of $\zeta$ and functoriality of
$\bullet$, the down-then-right path in \eqref{eq:sqcap_colinear} is
checked to be equal to 
$$
\xymatrix@R=8pt@C=10pt{
J\circ A \ar[r]^-{J\circ A\circ \Delta^2.\eta}&
J\circ A\circ (A\bullet A \bullet A)\ar[r]^-\zeta&
(J\circ (A\bullet A))\bullet (J\circ A\circ A)
\ar[r]^-{\zeta\bullet (J\circ A\circ A)}&\\
&(J\circ A)\bullet(J\circ A)\bullet(J\circ A\circ A)
\ar[r]^-{\raisebox{6pt}{
${}_{(J\circ A)\bullet(J\circ A)\bullet(J\circ\varepsilon.\mu)}$}}&
(J\circ A)\bullet(J\circ A)\bullet(J\circ J)
\ar[r]^-{\raisebox{6pt}{${}_{(J\circ A)\bullet\sqcap^R_\circ\bullet \varpi}$}}&
(J\circ A)\bullet(J\circ A).}
$$
Hence by commutativity of the diagram on page
\pageref{page_sqcap_colinear_1} in the Appendix, it is equal to
\begin{equation}\label{eq:auxi}
\xymatrix@C=10pt{
J\circ A \ar[r]^-{\raisebox{6pt}{${}_{J\circ A\circ \Delta.\eta}$}}&
J\circ A\circ (A\bullet A)\ar[r]^-\zeta&
(J\circ A)\bullet (J\circ A\circ A)
\ar[r]^-{\raisebox{6pt}{${}_{(J\circ A)\bullet (J\circ\mu)}$}}&
(J\circ A)\bullet (J\circ A)
\ar[r]^-{\raisebox{6pt}{${}_{(J\circ A)\bullet \sqcap^R_\circ}$}}&
(J\circ A)\bullet (J\circ A).}
\end{equation}
In the diagram on page \pageref{page_sqcap_colinear_1}, the region
labelled by $(\ast)$ is identical to the commutative diagram on page
\pageref{page_1.5} if we substitute in the latter one $X=Y=J$ and $Z=J\circ
A$.
On the other hand, also the right-then-down path in \eqref{eq:sqcap_colinear}
is equal to \eqref{eq:auxi} by commutativity of the diagram on page
\pageref{page_sqcap_colinear_2} in the Appendix. Commutativity of the region
marked by $(\ast)$ in the diagram on page \pageref{page_sqcap_colinear_2}
follows by commutativity of the diagram on page
\pageref{page_1.5} (substituting in it $X=J$ and $Y=Z=A$).
The vertical path on the right hand side of the diagram on page
\pageref{page_sqcap_colinear_2} is equal to $\sqcap^R_\circ \bullet(J\circ A)$
by \eqref{eq:sqcap_alt}.
This proves that $\sqcap^R_\circ$ is a morphism of right
$R_\circ$-comodules. Then so are its splitting mono- and epimorphisms with
respect to the coaction
$$
\xymatrix{
R_\circ \ar[r]^-{\iota^R_\circ}&
J\circ A \ar[r]^-{J\circ \Delta}&
J\circ (A\bullet A)\ar[r]^-\zeta&
(J\circ A)\bullet (J\circ A)\ar[r]^-{\pi^R_\circ\bullet\pi^R_\circ}&
R_\circ \bullet R_\circ\ =\ \Delta^R_\circ}
$$
on $R_\circ$, where the equality follows by using that $\pi^R_\circ$ is a
morphism of comonoids.
\end{proof}

\begin{lemma}\label{lem:comonoid_map}
For any right weak bimonoid $A$ in a duoidal category $(\mathcal M, \circ,
\bullet)$ in which idempotent morphisms split, there is a morphism of
comonoids $\xymatrix@C=15pt{A\ar[r]^-{\tau\circ A}& J\circ
A\ar[r]^-{\pi^R_\circ}& R_\circ .}$ 
\end{lemma}

\begin{proof}
Since $\pi^R_\circ:J\circ A \to R_\circ$ is a morphism of comonoids, we only
need to show that $\tau\circ A:A\to J\circ A$ is so. Its counitality follows
by the unitality of $\varpi$ and its comultiplicativity follows by
commutativity of the following diagram (in which the middle region commutes 
by the counitality of $\delta$).
$$
\xymatrix{
A\ar[r]^-\Delta\ar[d]_-{\tau \circ A}&
A\bullet A \ar[r]_-{\delta \circ (A\bullet A)}
\ar[d]^-{\tau \circ (A\bullet A)}
\ar@/^1.5pc/@{=}[rr]_-{\eqref{eq1.2}}&
(I\bullet I)\circ (A\bullet A)\ar[r]_-\zeta
\ar[d]^-{(\tau\bullet \tau)\circ (A\bullet A)}&
A\bullet A\ar[d]^-{(\tau \circ A)\bullet (\tau \circ A)}\\
J\circ A \ar[r]_-{J\circ \Delta}&
J\circ (A\bullet A)\ar@{=}[r]&
J\circ (A\bullet A)\ar[r]_-\zeta&
(J \circ A)\bullet (J \circ A)}
$$
\end{proof}

\begin{lemma}\label{lem:R_bimod}
For any right weak bimonoid $A$ in a duoidal category $(\mathcal M, \circ,
\bullet)$ in which idempotent morphisms split, $R_\circ$ is a $J$-$A$ bimodule.
\end{lemma}

\begin{proof}
Recall from \cite[page 11]{G. Bohm2011} that $R_\circ$ is a right $A$-module
via the action
$$
\xymatrix{
R_\circ \circ A \ar[r]^-{\iota^R_\circ \circ A}&
J\circ A \circ A \ar[r]^-{J\circ \mu}&
J\circ A \ar[r]^-{\pi^R_\circ}&
R_\circ .
}
$$
Symmetrically, it is a left $J$-module via
$$
\xymatrix{
J\circ R_\circ \ar[r]^-{J\circ \iota^R_\circ }&
J\circ J\circ A \ar[r]^-{\varpi \circ A}&
J\circ A \ar[r]^-{\pi^R_\circ}&
R_\circ .
}
$$
Unitality of both actions is evident. They are also associative and commute
with each other by the associativity of $\mu$ and of $\varpi$, together with
\eqref{eq:wbm_1.10} and the fact that $\sqcap^R_\circ$ is a morphism of left
$J$-modules (what follows from its form in \eqref{eq:sqcap_alt}), idempotency
of $\sqcap^R_\circ$ and functoriality of $\circ$.
\end{proof}

\begin{lemma}\label{lem:R_coeq}
For any right weak bimonoid $A$ in a duoidal category $(\mathcal M,\circ,
\bullet)$ in which idempotent morphisms split, there is a contractible
coequalizer of left $J$-modules 
$$
\xymatrix@C=35pt{
J\circ A\circ A \ar@<2pt>[r]^-{J\circ \mu}\ar@<-2pt>[r]_-{\vartheta^R}&
J\circ A \ar[r]^-{\pi^R_\circ}\ar@{-->}@/^1.5pc/[l]^-{J\circ A\circ \eta}&
R_\circ, \ar@{-->}@/^1.5pc/[l]^-{\iota^R_\circ}}
$$
where $\vartheta^R$ denotes the morphism
$$
\xymatrix@C=5pt@R=5pt{
&J\circ A\circ A\ar[rr]^-{J\circ A\circ \Delta}&&
J\circ A\circ (A\bullet A)\ar[r]^-\zeta&
(J\circ A)\bullet(J\circ A\circ A)
\ar[rr]^-{\raisebox{7pt}{${}_{(J\circ A)\bullet(J\circ \varepsilon.\mu)}$}}&&
(J\circ A)\bullet(J\circ J)
\ar[r]^-{\raisebox{7pt}{${}_{(J\circ A)\bullet \varpi}$}}&
J\circ A\\
=&
J\circ A\circ A\ar[rr]^-{J\circ A\circ \Delta}&&
J\circ A\circ (A\bullet A)\ar[r]^-{J\circ \zeta}&
J\circ((J\circ A)\bullet (A\circ A))
\ar[rr]^-{\raisebox{6pt}{${}_{J\circ((J\circ A)\bullet \varepsilon.\mu)}$}}&&
J\circ J\circ A \ar[r]^-{\varpi \circ A}&
J\circ A.}
$$
\end{lemma}

\begin{proof}
The given forms of $\vartheta^R$ are equal by \eqref{eq:theta_forms},
\eqref{eq1.1}, \eqref{eq1.2}, the counitality of $\Delta$, naturality of
$\zeta$ and functoriality of both monoidal products.

It follows by its form in \eqref{eq:sqcap_alt} that $\sqcap^R_\circ$ is a
morphism of left
$J$-modules. Hence so are $\pi^R_\circ$ and $\iota^R_\circ$ with respect to
the $J$-action on $R_\circ$ in Lemma \ref{lem:R_bimod}; proving that all
morphisms in the diagram are $J$-linear. By construction,
$\iota^R_\circ$ is a section of $\pi^R_\circ$ and $J\circ A\circ \eta$ is a
section of $J\circ \mu$. Both composites $\iota^R_\circ. \pi^R_\circ$ and
$\vartheta^R.(J\circ A\circ \eta)$ are equal to $\sqcap^R_\circ$ by virtue of
\eqref{eq:sqcap^R}. Finally, in the commutative diagram
$$
\xymatrix@R=15pt{
J\circ A \circ A \ar[rrr]^-{\vartheta^R}\ar[rd]^-{J\circ A\circ A\circ \eta}
\ar[dd]^-{J\circ \mu}&&&
J\circ A \ar[d]_-{J\circ A\circ \eta}
\\
& J\circ A \circ A\circ A\ar@{=}[r]\ar[d]^-{J\circ \mu \circ A}&
J\circ A \circ A\circ A\ar[r]^-{\vartheta^R \circ A}
\ar[d]_-{J\circ A \circ A\circ \Delta}&
J\circ A \circ A\ar[d]_-{J\circ A \circ \Delta}\\
J\circ A\ar[r]^-{J\circ A\circ \eta}\ar@/_4pc/[rrdddd]_-{\sqcap^R_\circ}&
J\circ A \circ A\ar[d]^-{J\circ A \circ \Delta}&
J\circ A \circ A \circ (A \bullet A)\ar[d]_-\zeta&
J\circ A \circ (A \bullet A)\ar[d]_-\zeta\\
&J\circ A \circ (A \bullet A)\ar[d]^-\zeta&
(J\circ A)\bullet (J\circ A \circ A\circ A)
\ar[r]^-{\raisebox{6pt}{${}_{(J\circ A)\bullet(\vartheta^R \circ A)}$}}
\ar[d]_-{(J\circ A)\bullet (J\circ \mu \circ A)}
\ar@{}[rddd]|-{(\LLC)}&
(J\circ A)\bullet (J\circ A \circ A)
\ar[d]_-{(J\circ A)\bullet (J\circ \varepsilon.\mu)}\\
&(J\circ A)\bullet (J\circ A \circ A)\ar@{=}[r]
\ar@{}[rdd]_-{\eqref{eq:sqcap^R}}&
(J\circ A)\bullet (J\circ A \circ A)
\ar[d]_-{(J\circ A)\bullet (J\circ \varepsilon.\mu)}&
(J\circ A)\bullet (J\circ J)\ar[dd]_-{(J\circ A)\bullet \varpi}\\
&&(J\circ A)\bullet (J\circ J)\ar[d]_-{(J\circ A)\bullet \varpi}\\
&& J\circ A \ar@{=}[r]&
J\circ A,}
$$
the right vertical is equal to $\sqcap^R_\circ=\iota^R_\circ.\pi^R_\circ$
by \eqref{eq:sqcap^R}.
Since $\iota^R_\circ$ is a (split) monomorphism, this proves that the solid
arrows in the diagram in the claim constitute a fork.
\end{proof}

\begin{lemma}\label{lem:R_J-sFmonoid}
For any right weak bimonoid $A$ in a duoidal category $(\mathcal M, \circ,
\bullet)$ in which idempotent morphisms split, $R_\circ$ is a separable
Frobenius monoid in the category of left $J$-modules.
\end{lemma}

\begin{proof}
We need to show that the structure morphisms \eqref{eq:R_sepFrob} are left
$J$-linear.
It is straightforward to check that $J\circ A$ is a comonoid in the category of
left $J$-modules, with comultiplication and counit
$$
\xymatrix@C=20pt{
J\circ A\ar[r]^-{J\circ \Delta}&
J\circ (A\bullet A)\ar[r]^-\zeta&
(J\circ A)\bullet (J\circ A) &\textrm{and}&
J\circ A\ar[r]^-{J\circ \varepsilon}&
J\circ J \ar[r]^-\varpi&
J.}
$$
Since $\pi^R_\circ:J\circ A \to R_\circ$ is a morphism of comonoids, also
$R_\circ$ is a comonoid in the category of left $J$-modules by Lemma
\ref{lem:R_coeq}.

It follows by Lemma \ref{lem:R_bimod} that the left $J$-action on $R_\circ$ is
a morphism of right $A$-modules. Hence it can be deduced from the naturality of
$\chi^R$, the naturality and associativity of $\zeta$ and the functoriality of
the monoidal product $\circ$, that $\chi^R_{R_\circ,R_\circ}$ is a morphism of
left $J$-modules. Since $\chi^R_{R_\circ,R_\circ}=\Delta^R_\circ.\mu^R_\circ$
and $\Delta^R_\circ$ is a morphism of left $J$-modules and also a (split)
monomorphism, $\mu^R_\circ$ is a morphism of $J$-modules too. Finally, the
unit $\eta^R_\circ$ is a morphism of left $J$-modules since $\pi^R_\circ$ is
so by Lemma \ref{lem:R_coeq}.
\end{proof}

\begin{lemma}\label{lem:mu_colinear}
Let $A$ be a right weak bimonoid in a duoidal category $(\mathcal
M,\circ,\bullet)$ in which idempotent morphisms split. Regard $A$ as a left
$R_\circ$-comodule via the coaction 
$$
\varrho:\xymatrix{
A \ar[r]^-\Delta&
A\bullet A \ar[r]^-{(\tau\circ A)\bullet A}&
(J\circ A) \bullet A\ar[r]^-{\pi^R_\circ \bullet A}&
R_\circ \bullet A}
$$
in terms of the comonoid morphism in Lemma \ref{lem:comonoid_map};
and regard $A\circ A$ as a left $R_\circ$-comodule via the coaction
$$
\xymatrix{
A\circ A \ar[r]^-{A \circ \varrho}&
A\circ (R_\circ \bullet A)\ar[r]^-\zeta&
(J\circ R_\circ)\bullet (A\circ A)\ar[r]^-{\gamma \bullet (A\circ A)}&
R_\circ\bullet (A\circ A)}
$$
in terms of the action $\gamma:J\circ R_\circ\to R_\circ$ (cf. Lemma
\ref{lem:R_bimod}). Then the multiplication $\mu:A\circ A \to A$ is a morphism
of left $R_\circ$-comodules.
\end{lemma}

\begin{proof}
The coaction on $A\circ A$ is counital and coassociative since $R_\circ$ is a
comonoid in the category of left $J$-modules (cf. Lemma
\ref{lem:R_J-sFmonoid}) and by the counitality and the coassociativity of
$\varrho$, using also an associativity axiom in \eqref{eq1.1}, naturality of
$\zeta$ and functoriality of the monoidal product $\bullet$. Then the claim
follows by commutativity of the diagram in the Appendix, on page
\pageref{page_mu_colinear}. In the diagram on page \pageref{page_mu_colinear},
the region marked by (A) commutes by the second explicit form of the morphism
$\vartheta^R$ in Lemma \ref{lem:R_coeq}, unitality of $\varpi$ and by
functoriality of the monoidal product $\bullet$. The region marked by (B)
commutes since $\pi^R_\circ$ coequalizes the parallel morphisms in Lemma
\ref{lem:R_coeq}. The region marked by (C) commutes since $\pi^R_\circ$ is a
morphism of left $J$-modules by Lemma \ref{lem:R_coeq}.
\end{proof}

Symmetric considerations apply to the Eilenberg-Moore categories of all weak
bi(co)monads in Corollary \ref{cor:wbms}.

\section{The base objects}\label{sec:base}

Let $A$ be a weak bimonoid in a duoidal category $(\mathcal M, \circ,\bullet)$
in which idempotent morphisms split. Applying the transformations $\circ$,
$\bullet$ and $\ast$ in Section \ref{sec:axioms}, and all of their composites
to the morphism $\sqcap^{R}_\circ$ in \eqref{eq:sqcap^R}, we obtain an eight
member family of idempotent morphisms fitting the diagram 
$$
\xymatrix@R=20pt{
& \overline\sqcap^L_\circ \ar@{<->}[rr]^-{\circ} \ar@{<-}[d] 
&& \sqcap^L_\circ\ar@{<->}[dd]^-{\ast} \\ 
\sqcap^R_\circ\ar@{<->}[rr]^(.7){\circ}\ar@{<->}[dd]_-{\ast}
\ar@{<->}[ru]^-{\bullet} 
& \ar@{->}[d]^-{\ast} 
& \overline\sqcap^R_\circ \ar@{<->}[dd]^(.3){\ast}
\ar@{<->}[ru]^-{\bullet}\\ 
& \overline\sqcap^R_\bullet \ar@{<-}[r] 
& \ar@{->}[r]^-{\bullet} 
& \sqcap^L_\bullet \\ 
\sqcap^R_\bullet\ar@{<->}[rr]_-{\bullet}\ar@{<->}[ru]_-{\circ} 
&& \overline\sqcap^L_\bullet . \ar@{<->}[ru]_-{\circ} 
}
$$
Their splittings define the respective objects $\overline L_\circ$, $L_\circ$,
$R_\circ$, $\overline R_\circ$, $\overline R_\bullet$, $L_\bullet$,
$R_\bullet$ and $\overline L_\bullet$ in $\mathcal M$. Each of them carries
a separable Frobenius monoid structure in the appropriate
monoidal category $(\mathcal M, \circ)$ or $(\mathcal M, \bullet)$.
Symmetric considerations to those in Section \ref{sec:modcat} apply to
them. For example, by symmetric versions of Lemma
\ref{lem:comonoid_map}, there are comonoid morphisms from $A$ to
$\overline L_\circ$, $L_\circ$, $R_\circ$ and $\overline R_\circ$; and there
are monoid morphisms from $\overline R_\bullet$, $L_\bullet$,
$R_\bullet$ and $\overline L_\bullet$ to $A$.

If $A$ is a weak bimonoid in a braided monoidal category --- regarded as a
duoidal category --- whose idempotent morphisms split, then $R_\circ$,
$R_\bullet$ and $\overline R_\bullet$ become isomorphic to the `right' or
`source' Frobenius monoid, while $\overline R_\circ$ becomes isomorphic to the
opposite monoid and opposite comonoid. Symmetrically, $L_\circ$,
$L_\bullet$ and $\overline L_\bullet$ become isomorphic to the `left' or
`target' Frobenius monoid, while $\overline L_\circ$ becomes isomorphic to the
opposite monoid and opposite comonoid. Thus all the eight base Frobenius
monoids become (anti-)isomorphic (cf. \cite{C. Pastro}). 
The aim of this section is to relate these objects (that we call the {\em
`base objects'} of $A$) in our more general setting.

Applying the isomorphism in \eqref{eq:R_Rbar_iso} to the weak
bi(co)monads in \eqref{eq:induced_wbm}, we obtain the following.

\begin{proposition}\label{prop:bar_base}
For any weak bimonoid $A$ in a duoidal category $(\mathcal M, \circ,\bullet)$
in which idempotent morphisms split, the following objects of $\mathcal M$
(defined above) are pairwise isomorphic.
\begin{itemize}
\item[{(1)}] $L_\bullet \cong \overline L_\bullet$ as monoids and
 comonoids in $(\mathcal M, \circ)$.
\item[{(2)}] $R_\bullet \cong \overline R_\bullet$ as monoids and
 comonoids in $(\mathcal M, \circ)$.
\item[{(3)}] $L_\circ \cong \overline R_\circ$ as monoids and
 comonoids in $(\mathcal M, \bullet)$.
\item[{(4)}] $R_\circ \cong \overline L_\circ$ as monoids and
 comonoids in $(\mathcal M, \bullet)$.
\end{itemize}
\end{proposition}

Note that by splitting certain idempotent morphisms, the objects
occurring in Proposition \ref{prop:bar_base} are defined only up-to
isomorphism. Hence without any loss of generality, we may identify the
isomorphic objects in parts (1)-(4). We will do so throughout the paper,
replacing the objects written on the right hand side in parts (1)-(4) of
Proposition \ref{prop:bar_base} with their isomorphic copies on the left hand
side.

\begin{proposition}\label{prop:bimodiso}
For any weak bimonoid $A$ in a duoidal category $(\mathcal M, \circ,\bullet)$
in which idempotent morphisms split, the following objects of $\mathcal M$
(defined above) are pairwise isomorphic.
\begin{itemize}
\item[{(1)}] $L_\circ \cong L_\bullet\circ J$ as $L_\bullet$-$J$ bimodules.
\item[{(2)}] $L_\circ \cong R_\bullet\circ J$ as $R_\bullet$-$J$ bimodules.
\item[{(3)}] $R_\circ \cong J\circ R_\bullet$ as $J$-$R_\bullet$ bimodules.
\item[{(4)}] $R_\circ \cong J\circ L_\bullet$ as $J$-$L_\bullet$ bimodules.
\item[{(5)}] $L_\bullet \cong L_\circ\bullet I$ as $L_\circ$-$I$ bicomodules.
\item[{(6)}] $L_\bullet \cong R_\circ\bullet I$ as $R_\circ$-$I$ bicomodules.
\item[{(7)}] $R_\bullet \cong I\bullet R_\circ$ as $I$-$R_\circ$ bicomodules.
\item[{(8)}] $R_\bullet \cong I\bullet L_\circ$ as $I$-$L_\circ$ bicomodules.
\end{itemize}
\end{proposition}
\begin{proof}
We only prove part (1), all other claims follow from it applying the
transformations $\circ$, $\bullet$, $\ast$ and their composites. Consider 
the morphisms 
$
\xymatrix{
\phi: (A\bullet I)\circ J\ar[r]^-{(A\bullet \tau)\circ J}&
A\circ J}
$ and
$$
\xymatrix@C=20pt @R=8pt{
\phi':A\circ J\ar[r]^-{\delta \circ A\circ J}&
(I\bullet I)\circ A\circ J\ar[r]^-{(\Delta.\eta\bullet I)\circ A\circ J}&
(A\bullet A\bullet I)\circ A\circ J\ar[r]^-\zeta&\\
&(A\circ A\circ J)\bullet ((A\bullet I)\circ J)
\ar[r]^-{\raisebox{7pt}{${}_{(\varepsilon.\mu \circ J)
\bullet ((A\bullet I)\circ J)}$}}&
(J\circ J)\bullet ((A\bullet I)\circ J)
\ar[r]^-{\raisebox{7pt}{${}_{\varpi \bullet ((A\bullet I)\circ J)}$}}&
(A\bullet I)\circ J.}
$$
By functoriality of both monoidal products, naturality of $\zeta$
and counitality of $\delta$, it follows that $\phi.\phi'=\sqcap^L_\circ$. On the
other hand, by commutativity of
$$
\xymatrix@C=45pt @R=20pt{
(A\bullet I)\circ J\ar[r]^-{(A\bullet \tau)\circ J}
\ar[d]_-{\delta \circ (A\bullet I)\circ J}&
A\circ J\ar[d]^-{\delta \circ A\circ J}\\
(I\bullet I)\circ (A\bullet I)\circ J
\ar[d]_-{(\Delta.\eta \bullet I)\circ (A\bullet I)\circ J}&
(I\bullet I)\circ A \circ J
\ar[d]^-{(\Delta.\eta \bullet I)\circ A\circ J}\\
(A\bullet A\bullet I)\circ (A\bullet I)\circ J
\ar[d]_-{\zeta \circ J}
\ar[r]^-{(A\bullet A\bullet I)\circ (A\bullet \tau)\circ J}
_-{\stackrel{\eqref{eq:coh}}=(A\bullet A\bullet I)\circ \zeta}
\ar@{}[rd]|-{\raisebox{-10pt}{${}_{\eqref{eq1.1}}$}}&
(A\bullet A\bullet I)\circ A\circ J\ar[d]^-\zeta\\
((A\circ A)\bullet A\bullet I)\circ J\ar[r]_-\zeta
\ar[d]_-{(\varepsilon.\mu\bullet A\bullet I)\circ J}&
(A\circ A\circ J)\bullet ((A\bullet I)\circ J)
\ar[d]^-{(\varepsilon.\mu\circ J)\bullet ((A\bullet I)\circ J)}\\
(A\bullet I)\circ J\ar[r]^-\zeta\ar@{=}[rd]^-{\eqref{eq1.2}}&
(J \circ J)\bullet ((A\bullet I)\circ J)
\ar[d]^-{\varpi \bullet ((A\bullet I)\circ J)}\\
& (A\bullet I)\circ J,}
$$
also $\phi'.\phi=\sqcap^L_\bullet \circ J$. Hence $\phi$ and $\phi'$
(co)restrict to the stated isomorphisms
$$
\xymatrix@C=10pt{
L_\bullet\circ J \ar[r]^-{\iota^L_\bullet\circ J}&
(A\bullet I)\circ J \ar[r]^-\phi&
A\circ J \ar[r]^-{\pi^L_\circ}&
L_\circ&
\textrm{and}&
L_\circ \ar[r]^-{\iota^L_\circ}&
A\circ J \ar[r]^-{\phi'}&
(A\bullet I)\circ J \ar[r]^-{\pi^L_\bullet\circ J}&
L_\bullet\circ J.}
$$

By a symmetric version of Lemma \ref{lem:R_bimod}, $L_\circ$ is an $A$-$J$
bimodule. Hence it is an $L_\bullet$-$J$ bimodule via restriction along the
monoid morphism
$\xymatrix@C=15pt{
L_\bullet\ar[r]^-{\iota^L_\bullet}&
A\bullet I \ar[r]^-{A\bullet \tau}&
A}$ (cf. a symmetric form of Lemma \ref{lem:comonoid_map})
. The morphisms $\phi$ and $\sqcap^L_\bullet \circ J$ are evidently right
$J$-module maps and so is $\sqcap^L_\circ$ by a symmetric version of Lemma
\ref{lem:R_coeq}. Hence the induced isomorphism $L_\bullet \circ J\to L_\circ$
is also a morphism of right $J$-modules. Since $L_\bullet \circ \pi^L_\bullet
\circ J$ is a (split) epimorphism, the isomorphism $L_\bullet \circ J
\to L_\circ$ is also a morphism of left $L_\bullet$-modules by
commutativity of
$$
\xymatrix@R=15pt{
L_\bullet \circ L_\bullet \circ J \ar[r]^-{L_\bullet\circ \iota^L_\bullet\circ J}
\ar[ddddd]_-{\mu^L_\bullet\circ J}&
L_\bullet \circ (A\bullet I) \circ J
\ar[r]^-{L_\bullet \circ (A\bullet \tau) \circ J}&
L_\bullet \circ A \circ J\ar[r]^-{L_\bullet \circ \pi^L_\circ}&
L_\bullet \circ L_\circ \ar[d]^-{\iota^L_\bullet \circ L_\circ}\\
\ar@{}[rddd]|-{(C)}
&L_\bullet \circ (A\bullet I) \circ J
\ar@{->>}[lu]_(.4){L_\bullet \circ \pi^L_\bullet \circ J}
\ar[r]^-{L_\bullet \circ (A\bullet \tau) \circ J}
\ar[d]^(.4){\iota^L_\bullet \circ (A\bullet I) \circ J}&
L_\bullet \circ A \circ J\ar@{->>}[ru]^(.4){L_\bullet \circ \pi^L_\circ}
\ar[d]_(.4){\iota^L_\bullet \circ A \circ J}&
(A\bullet I) \circ L_\circ\ar[d]^-{(A\bullet \tau) \circ L_\circ}\\
&(A\bullet I) \circ (A\bullet I) \circ J
\ar[r]^-{\raisebox{5pt}{${}_{(A\bullet I) \circ (A\bullet \tau) \circ J}$}}
\ar[d]^-{\zeta\circ J}
\ar@{}[rdd]|-{(B)}&
(A\bullet I) \circ A \circ J
\ar[d]_-{(A\bullet \tau) \circ A \circ J}&
A \circ L_\circ\ar[d]^-{A \circ \iota^L_\circ}\\
& ((A\circ A)\bullet I)\circ J\ar[d]^-{(\mu\bullet I)\circ J}&
A\circ A \circ J\ar@{->>}[ru]^(.4){A\circ \pi^L_\circ}
\ar[r]_-{A\circ \sqcap^L_\circ}\ar[d]_-{\mu\circ J}\ar@{}[rdd]|-{(A)}&
A\circ A \circ J\ar[d]^-{\mu\circ J}\\
& (A\bullet I) \circ J\ar@{->>}[ld]^(.3){\pi^L_\bullet \circ J}
\ar[r]^-{(A\bullet \tau) \circ J}&
A\circ J\ar@{->>}[rd]_(.3){\pi^L_\circ}&
A\circ J\ar[d]^-{\pi^L_\circ}\\
L_\bullet \circ J\ar[r]_-{\iota^L_\bullet \circ J}&
(A\bullet I) \circ J \ar[r]_-{(A\bullet \tau) \circ J}&
A\circ J \ar[r]_-{\pi^L_\circ}&
L_\circ .}
$$
The region marked by $(A)$ commutes by \eqref{eq:wbm_1.10}. The region marked
by $(B)$ commutes since $A\bullet \tau:A\bullet I\to A$ is a monoid morphism
(by a dual version of the reasoning in the proof of Lemma
\ref{lem:comonoid_map}). The region marked by $(C)$ commutes since
$\pi^L_\bullet$ is a morphism of left $L_\bullet$-modules by a symmetric
version of Lemma \ref{lem:pi_colinear}.
\end{proof}

\section{Hopf modules}\label{sec:rel_Hopf}

Hopf modules over a (weak) bialgebra are both modules and comodules with an
appropriate compatibility condition between the action and the coaction. The
compatibility condition is most conveniently formulated in terms of a `(weak)
mixed distributive law' (called a `(weak) entwining structure' in
\cite{CaeDeGr}). Our definition of Hopf modules over weak bimonoids in a
duoidal category in this section, uses this language. 

\begin{proposition}\label{prop:Hopf_d_law}
For any weak bimonoid $A$ in a duoidal category $(\mathcal M,\circ,\bullet)$,
there is a weak mixed distributive law between the induced monad $(-)\circ A$
and comonad $(-)\bullet A$ on $\mathcal M$.
\end{proposition}

\begin{proof} We construct the desired weak mixed distributive law putting
$$\psi_M:\xymatrix@C=10pt {
(M\bullet A)\circ A \ar[rr]^-{(M\bullet A)\circ \Delta}
&& (M\bullet A)\circ (A\bullet A)\ar[r]^-{\zeta}
& (M\circ A)\bullet(A\circ A)\ar[rr]^-{(M\circ A)\bullet\mu}
&& (M\circ A)\bullet A,
}$$
for any object $M\in \mathcal M$.
The compatibility
$$
\xymatrix{
(M\bullet A)\circ A\circ A\ar[r]^-{\psi_M\circ A}
\ar[d]_-{(M\bullet A)\circ \mu}&
((M\circ A)\bullet A)\circ A\ar[r]^-{\psi_{M\circ A}}&
(M\circ A\circ A)\bullet A\ar[d]^-{(M\circ \mu)\bullet A}\\
(M\bullet A)\circ A \ar[rr]_-{\psi_M}&&
(M\circ A)\bullet A}
$$
with the multiplication follows from the first axiom \eqref{eq:WB} of weak
bimonoids, the associativity axioms \eqref{eq1.1} of a duoidal category and
associativity of the monoid $A$, together with the naturality of $\zeta$
and functoriality of both monoidal products. (This condition is proven
in the same way as in the non-weak case). The compatibility
$$
\xymatrix{
M\bullet A
\ar[r]^-{\raisebox{6pt}{${}_{M\bullet \Delta}$}}
\ar[d]_-{(M\bullet A)\circ \eta}&
M\bullet A \bullet A
\ar[r]^-{\raisebox{6pt}{${}_{((M\bullet A)\circ \eta) \bullet A}$}}&
((M\bullet A)\circ A) \bullet A
\ar[r]^-{\raisebox{6pt}{${}_{\psi_M\bullet A}$}}&
(M\circ A)\bullet A \bullet A
\ar[d]^-{(M\circ A)\bullet \varepsilon \bullet A}\\
(M\bullet A)\circ A\ar[rrr]_-{\psi_M}&&&
(M\circ A)\bullet A}
$$
with the unit is proven in the Appendix on page \pageref{page_wdl_unit}.
The compatibility conditions with the comultiplication and the counit are
obtained by applying the transformation $\ast$ to the above diagrams. Hence
they follow by symmetry.
\end{proof}

Applying the duality transformations $\bullet$, $\circ$ and their
composite to the weak distributive law $\psi$ in Proposition
\ref{prop:Hopf_d_law}, we obtain a four member family of weak distributive
laws between the various induced (co)monads. In the rest of the paper we
always work with $\psi$ in Proposition \ref{prop:Hopf_d_law} but certainly
symmetric considerations apply to all of its dual counterparts.

Over the weak mixed distributive law $\psi$ in Proposition
\ref{prop:Hopf_d_law}, we may consider the mixed modules
\cite{Bohm:WTM,BLS:2cat_wdl} in the following sense. 

\begin{definition}\label{def:Hopf_mod}
Mixed modules over the weak mixed distributive law $\psi$ in Proposition
\ref{prop:Hopf_d_law} are called {\em Hopf modules} over the weak bimonoid
$A$. Explicitly, this means an object $M$ in $\mathcal M$ which is both a
$(-)\circ A$-module with action $\gamma$ and a $(-)\bullet
A$-comodule with coaction $\varrho$ such that the following diagram
commutes.
\begin{equation}
\xymatrix @C=10pt{
M\circ A \ar[rrr]^{\gamma} \ar[d]_-{\varrho\circ A}
&&& M \ar[rrr]^-{\varrho}
&&& M\bullet A \\
(M\bullet A)\circ A\ar[rr]^-{(M\bullet A)\circ \Delta}
\ar@/_1.5pc/[rrrrrr]_-{\psi_M}&&
(M\bullet A)\circ(A\bullet A) \ar[rr]^-{\zeta}&&
(M\circ A)\bullet(A\circ A)\ar[rr]^-{(M\circ A)\circ \mu}&&
(M\circ A)\bullet A
 \ar[u]_-{\gamma\bullet A}}
\label{eq2.2}
\end{equation}
Morphisms of Hopf modules are morphisms of $(-)\bullet A$-comodules
and $(-)\circ A$-modules. The category of Hopf modules is
denoted by $\mathcal{M}^A_A$.
\end{definition}

By the weak bimonoid axiom \eqref{eq:WB}, $A$ is a Hopf module over itself via
the multiplication and the comultiplication. 

For any monad $t$ and comonad $c$ on the same category $\mathcal M$, it
follows by \cite[Corollary 5.11 and Proposition 6.3]{PowWat} that (non-weak)
mixed distributive laws $tc\to ct$ are in a bijective correspondence with
{\em liftings} $\overline c$ of the comonad $c$ to the Eilenberg-Moore
category $\mathcal M^t$ of the monad $t$. This means that the functor
$\overline c$ renders commutative the diagram 
\begin{equation}\label{eq:lift}
\xymatrix{
\mathcal M^t\ar[r]^-{\overline c}\ar[d]_-{u^t}&
\mathcal M^t\ar[d]^-{u^t}\\
\mathcal M\ar[r]_-c&
\mathcal M,}
\end{equation}
involving the forgetful functor $u^t$. Moreover, the comultiplication
$\overline \delta$ and the counit $\overline \varepsilon$ of the comonad
$\overline c$ are related to the comultiplication $\delta$ and the counit
$\varepsilon$ of $c$ via $u^t\overline \delta=\delta u^t$ and $u^t\overline
\varepsilon=\varepsilon u^t$.
 
In a similar way, by \cite[Proposition 5.7]{Bohm:WTM}, whenever idempotent
morphisms in $\mathcal M$ split, a weak mixed distributive law $tc\to ct$
determines a {\em weak lifting} $\overline c$ of $c$ to $\mathcal
M^t$. This means that the diagram in \eqref{eq:lift} does not need strictly
commute. Instead, there are natural transformations $\iota:u^t\overline c \to
c u^t$ and $\pi:c u^t \to u^t\overline c$ such that their composite
$\pi.\iota$ is the identity (they are obtained by splitting an idempotent
natural transformation canonically associated to the weak mixed distributive
law). Moreover, $(u^t,\iota)$ is a comonad morphism from $\overline c$ to $c$
in the sense of \cite{R. Street}. 

In the situation of Proposition \ref{prop:Hopf_d_law}, this means the
following. Associated to the weak mixed distributive law $\psi$ in Proposition
\ref{prop:Hopf_d_law}, for any $(-)\circ A$-module $(Q,\gamma)$
there is an idempotent morphism
$$
\xymatrix@C=45pt{
Q\bullet A \ar[r]^-{(Q\bullet A)\circ \eta}&
(Q\bullet A)\circ A\ar[r]^-{\psi_Q}&
(Q\circ A)\bullet A\ar[r]^-{\gamma\bullet A}&
Q\bullet A.}
$$
In light of the explicit form of $\psi_Q$, it is equal to $\chi^R_{Q,A}$ in
\eqref{chi}. So whenever idempotent morphisms in $\mathcal M$ split,
it splits through the $R_\circ$-module product $Q\bullet_{R_\circ} A$. 
This defines a comonad $(-)\bullet_{R_\circ} A$ on the category $\mathcal M_A$
of $(-)\circ A$-modules, which is the weak lifting of the comonad
$(-)\bullet A$ on $\mathcal M$, induced by the weak mixed distributive law
$\psi$ in Proposition \ref{prop:Hopf_d_law} via the correspondence in
\cite[Proposition 5.7]{Bohm:WTM}. As proven in \cite[Proposition
3.7]{Bohm:WTM}, the Eilenberg-Moore category of comodules over this weakly
lifted comonad $(-)\bullet_{R_\circ} A$ on $\mathcal M_A$, is isomorphic to
the category $\mathcal{M}^A_A$ of mixed modules over $\psi$. By these
considerations, whenever idempotent morphisms in $\mathcal M$ split, the
forgetful functor $\mathcal{M}^A_A\to \mathcal{M}_A$ is comonadic. 

\begin{proposition}\label{prop:rel_Hopf_lifting}
Let $A$ be a weak bimonoid in a duoidal category $(\mathcal
M,\circ,\bullet)$.
Then the functor $(-)\circ A:\mathcal M \to {\mathcal M}_A$
lifts to $\mathcal M^I\rightarrow \mathcal M^A_A$ (where $I$ stands for the
$\circ$-monoidal unit regarded as a comonoid in $(\mathcal M,\bullet)$, and
$\mathcal M^I$ denotes the category of $(-)\bullet
I$-comodules). That is, there is a functor $\mathcal M^I\rightarrow
\mathcal M^A_A$ rendering commutative the diagram 
$$
\xymatrix{
\mathcal M^I\ar@{-->}[r]\ar[d]&
\mathcal M^A_A\ar[d]\\
\mathcal M\ar[r]_-{(-)\circ A}&
\mathcal M_A,}
$$
in which the vertical arrows denote the forgetful functors.
\end{proposition}

\begin{proof}
The object map of the desired functor is provided by the action and the
coaction
$$
\xymatrix{
Z\circ A \circ A \ar[r]^-{Z\circ \mu}&Z\circ A&&
Z\circ A \ar[r]^-{\varrho \circ \Delta}&
(Z\bullet I) \circ (A \bullet A)\ar[r]^-\zeta&
(Z\circ A) \bullet A,}
$$
for any $(-)\bullet I$-comodule $(Z,\varrho)$. The action is associative and
unital by the associativity and unitality of the multiplication. The coaction
is coassociative by coassociativity of $\Delta$ and $\varrho$ and the axioms
\eqref{eq1.1} and \eqref{eq1.2}, together with the naturality of $\zeta$. It
is counital by \eqref{eq:coh}, the counitality of $\Delta$ and $\varrho$, and
the naturality of $\zeta$. The compatibility condition \eqref{eq2.2} follows
by the weak bimonoid axiom \eqref{eq:WB}, an associativity axiom in
\eqref{eq1.1}, and the naturality of $\zeta$ and functoriality of the monoidal
product $\circ$. Obviously, for any $(-)\bullet I$-comodule morphism $f$,
$f\circ A$ becomes a morphism of Hopf modules with respect to these actions
and coactions.
\end{proof}

Whenever idempotent morphisms in $\mathcal M$ split --- hence $\mathcal M^A_A$
is isomorphic to the Eilenberg-Moore category of the comonad
$(-)\bullet_{R_\circ} A$ on $\mathcal M_A$ --- it follows by \cite[Corollary
 5.11]{PowWat} that the lifted functor in Proposition
\ref{prop:rel_Hopf_lifting} corresponds to a comonad morphism (in the sense of
\cite{R. Street}), $\lambda^0:((-)\bullet I)\circ A \to ((-)\circ 
A) \bullet_{R_\circ} A$, between functors $M\to M_A$. Applying a dual
form of \cite[Lemma 3.6]{Bohm:WTM} --- yielding the explicit correspondence
between the $A$-Hopf modules and the comodules of the comonad
$(-)\bullet_{R_\circ} A$ on $\mathcal M_A$ ---, for any object $M$ of
$\mathcal M$ the explicit form of $\lambda^0_M$ comes out as 
\begin{equation}\label{eq:lambda^0}
\lambda^0_M:\xymatrix{
(M\bullet I)\circ A\ar[r]^-{(M\bullet I)\circ \Delta}&
(M\bullet I)\circ (A\bullet A)\ar[r]^-\zeta&
(M \circ A)\bullet A\ar[r]^-{\pi^R_{M\circ A,A}}&
(M\circ A)\bullet_{R_\circ} A.}
\end{equation}
In other words, $\lambda^0_M$ is the unique morphism for which
$$
\xymatrix@C=65pt{
(M\bullet I)\circ A\ar@{-->}[r]^-{\lambda^0_M}
\ar[d]_-{(M\bullet I)\circ \Delta}&
(M\circ A) \bullet_{R_\circ} A\ar[d]^-{\iota^R_{M\circ A,A}}\\
(M\bullet I)\circ (A\bullet A)\ar[r]_-\zeta&
(M\circ A) \bullet A}
$$
commutes.
Both definitions of $\lambda^0_M$ are equivalent since
$$
\xymatrix{
(M\bullet I)\circ A \ar[r]^-{(M\bullet I)\circ \Delta}\ar@{=}[d]&
(M\bullet I)\circ (A\bullet A)\ar[r]^-\zeta&
(M\circ A)\bullet A\ar[d]^-{\chi^R_{(M\circ A),A}}\\
(M\bullet I)\circ A \ar[r]_-{(M\bullet I)\circ \Delta}&
(M\bullet I)\circ (A\bullet A)\ar[r]_-\zeta&
(M\circ A)\bullet A}
$$
commutes by \eqref{eq1.1}, \eqref{eq:WB} and unitality of the monoid $A$,
together with the naturality of $\zeta$ and functoriality of $\circ$. 
As in \cite[Proposition 1.1]{GomTor}, the comonad morphism $\lambda^0$
determines a ({\em `Galois-type'}\,) comonad morphism 
$$
\xymatrix@C=40pt{
\beta^0_Q:(Q\bullet I)\circ A\ar[r]^-{\lambda^0_Q}&
(Q\circ A) \bullet_{R_\circ} A\ar[r]^-{\gamma\bullet_{R_\circ} A}&
Q\bullet_{R_\circ} A,}
$$
for any right $A$-module $(Q,\gamma)$. It is in fact the unique morphism for
which
$$
\xymatrix{
(Q\bullet I)\circ A\ar@{-->}[rr]^-{\beta^0_Q}
\ar[d]_-{(Q\bullet I)\circ \Delta}&&
Q\bullet_{R_\circ} A\ar[d]^-{\iota^R_{Q,A}}\\
(Q\bullet I)\circ(A\bullet A)\ar[r]_-\zeta&
(Q\circ A) \bullet A\ar[r]_-{\gamma\bullet A}&
Q\bullet A}
$$
commutes. 

Regard the symmetric monoidal category of vector spaces as a duoidal
category. Clearly, its idempotent morphisms split. A weak bimonoid $A$ in this
duoidal category is the usual notion of weak bialgebra in
\cite{BSz,Nill:WBA,BNSz:WHAI}. By \cite[Section 36.16]{BrzWis}, the
fundamental theorem of Hopf modules holds for $A$ if and only if it is a weak
Hopf algebra. In this case, the fundamental theorem asserts that a certain
comparison functor, from the category of modules over the `left' or `target'
algebra to the category of $A$-Hopf modules, is an equivalence. In order to
obtain the generalization of this comparison functor in our setting of duoidal
categories, it is not enough to consider the lifted functor in Proposition
\ref{prop:rel_Hopf_lifting}, we need another lifting.

If $A$ is a weak bimonoid in a duoidal category $(\mathcal M,\circ,\bullet)$
in which idempotent morphisms split, then it follows by a symmetric version of
Lemma \ref{lem:R_coeq} that $\overline L_\bullet$ (and hence by Proposition
\ref{prop:bar_base}~(1) also the isomorphic object $L_\bullet$ with which we
identified it) fits the equalizer
\begin{equation}\label{eq:L_eq}
\xymatrix{
L_\bullet \ar[r]^-{\iota^L_\bullet}&
A\bullet I \ar@<2pt>[rrrrr]^-{\Delta\bullet I}
\ar@<-2pt>[rrrrr]_-{\vartheta^L:=(\mu\bullet A\bullet I).(\zeta\bullet I).
(((A\bullet I)\circ \Delta.\eta)\bullet I).(A\bullet \delta)}&&&&&
A\bullet A\bullet I}
\end{equation}
in $\mathcal M^I$.

We know from (a symmetric counterpart of) Lemma \ref{lem:R_J-sFmonoid} that
$L_\bullet$ is a separable Frobenius monoid in $\mathcal M^I$. In particular,
$L_\bullet$ is a monoid in $\mathcal M^I$, so we can take its category of
modules ${\mathcal M}^I_{L_\bullet}$. Equivalently, ${\mathcal
M}^I_{L_\bullet}$ is the category of comodules over the comonad
$(-)\bullet I$ on ${\mathcal M}_{L_\bullet}$. Thus it comes equipped
with a forgetful functor $V^I:{\mathcal M}^I_{L_\bullet}\to {\mathcal
M}_{L_\bullet}$, with right adjoint $(-)\bullet I$.

By a symmetric version of Lemma \ref{lem:comonoid_map}, there is a monoid
morphism $\xymatrix@C=10pt{\omega:
L_\bullet\ar[r]^-{\iota^L_\bullet}&
A\bullet I \ar[r]^-{A\bullet \tau}&
A}$ in $(\mathcal M, \circ)$. It induces a left $L_\bullet$-action on any left
$A$-module and hence it induces in particular a left $L_\bullet$-action
$\alpha$ on $A$. For any right $L_\bullet$-module $(P,\xi)$, the
$L_\bullet$-module tensor product $P\circ_{L_\bullet} A$ is defined ---
if it exists --- as the coequalizer
\begin{equation}\label{eq:pi_L}
\xymatrix@C=45pt{
P\circ L_\bullet \circ A \ar@<2pt>[r]^-{\xi \circ A}
\ar@<-2pt>[r]_-{P \circ \alpha}&
P\circ A\ar[r]^-{\pi^L_{P,A}}&
P\circ_{L_\bullet} A}
\end{equation}
in $\mathcal M$.
Recall (e.g. from \cite[Section 3.1]{C. Pastro}) that, since $L_\bullet$ is a
separable Frobenius monoid (in $(\mathcal M,\circ)$), $P\circ_{L_\bullet} A$
can be obtained also by splitting the idempotent morphism 
\begin{equation}\label{eq:chi_L}
\xymatrix@C=45pt{
P\circ A \ar[r]^-{P\circ \Delta^L_\bullet .
\eta^L_\bullet \circ A}&
P\circ L_\bullet \circ L_\bullet \circ A \ar[r]^-{\xi \circ \alpha}&
P\circ A }
\end{equation}
as $\xymatrix@C=30pt{P\circ A\ar@{->>}[r]|-{\,\pi^L_{P,A}}&
P\circ_{L_\bullet} A\ar@{>->}[r]|-{\iota^L_{P,A}}&P\circ A.}$
Hence by the assumption that idempotent morphisms in $\mathcal M$ split, the
$L_\bullet$-module tensor product $P\circ_{L_\bullet} A$ exists, and it is
preserved by any functor. This defines a functor
$\omega^*:=(-)\circ_{L_\bullet} A: {\mathcal M}_{L_\bullet}\to
{\mathcal M}_A$. It is left adjoint of the functor $\omega_*:{\mathcal M}_A \to
{\mathcal M}_{L_\bullet}$ defined by regarding any $A$ module as an
$L_\bullet$-module via $\omega$ (see e.g. \cite{Par}).
 
\begin{lemma}\label{lem:lambda_fork}
Let $A$ be a weak bimonoid in a duoidal category $(\mathcal M,\circ, \bullet)$
whose idempotent morphisms split. Then for any right $L_\bullet$-module
$(P,\xi)$, 
$$
\xymatrix@C=35pt{
(P\bullet I)\circ {L_\bullet}\circ A
\ar@<2pt>[r]^-{\gamma \circ A}
\ar@<-2pt>[r]_-{(P\bullet I)\circ\alpha}
& (P\bullet I)\circ A
\ar[rr]^-{(\pi^L_{P,A}\bullet_{R_\circ} A).\lambda^0_P}&&
(P\circ_{{L_\bullet}} A)\bullet_{R_\circ} A}
$$
is a fork, where $\lambda^0$ is the comonad morphism in \eqref{eq:lambda^0},
$\alpha=\mu.(\omega\circ A)$ is the $L_\bullet$-action on $A$ induced by
the monoid morphism $\omega=(A\bullet \tau).\iota^L_\bullet
:L_\bullet \to A$ and $\gamma$ is the ${L_\bullet}$-action on $P\bullet I$,
given in terms of the $I$-coaction $\varrho:L_\bullet \to L_\bullet \bullet I$
from a dual form of Lemma \ref{lem:R_bimod} as
\begin{equation}\label{eq:gamma}
\xymatrix@C=12pt{
(P\bullet I)\circ {L_\bullet}\ar[rr]^-{(P \bullet I)\circ \varrho}
&& (P\bullet I)\circ ({L_\bullet}\bullet I)\ar[r]^-\zeta
& (P\circ {L_\bullet})\bullet I\ar[r]^-{\xi\bullet I}
& P\bullet I.}
\end{equation}
\end{lemma}

\begin{proof}
The proof can be found in the Appendix, on page \pageref{page_lambda_fork}. In
the diagram on page \pageref{page_lambda_fork}, the region marked by ``Lemma
\ref{lem:mu_colinear}'' commutes since by a symmetric version of Lemma
\ref{lem:mu_colinear}, $\Delta:A \to A\bullet A$ is a morphism of left
$\overline L_\bullet\cong L_\bullet$-modules (for the isomorphism see Proposition
\ref{prop:bar_base}).
\end{proof}

\begin{proposition} \label{prop:lambda}
Let $A$ be a weak bimonoid in a duoidal category $(\mathcal M,\circ,\bullet)$
in which idempotent morphisms split. Then the functor
$\omega^*=(-)\circ_{L_\bullet} A: {\mathcal M}_{L_\bullet}\to {\mathcal M}_A$
above lifts to ${\mathcal M}^I_{L_\bullet}\to {\mathcal M}_A^A$. That
is, there is a functor ${\mathcal M}^I_{L_\bullet}\to {\mathcal M}_A^A$,
rendering commutative the diagram
$$
\xymatrix{
{\mathcal M}^I_{L_\bullet}\ar@{-->}[r]\ar[d]& 
{\mathcal M}_A^A\ar[d]\\
{\mathcal M}_{L_\bullet}\ar[r]_-{\omega^*}& 
{\mathcal M}_A,}
$$
in which the vertical arrows denote the (comonadic) forgetful functors.
\end{proposition}

\begin{proof}
The category $\mathcal M^A_A$ of Hopf modules is isomorphic to the
Eilenberg-Moore category of comodules over the comonad $(-)\bullet_{R_\circ}
A$ on $\mathcal M_A$. Hence using the bijection between liftings and
comonad morphisms (see e.g. \cite[Corollary 5.11]{PowWat}), the proof amounts
to constructing a morphism $\lambda_P: (P\bullet I)\circ_{L_\bullet} A \to (P
\circ_{{L_\bullet}} A)\bullet_{R_\circ} A$ in ${\mathcal M}_A$, for any right
${L_\bullet}$-module $(P,\xi)$; and showing that $\lambda$ is in fact a
comonad morphism (in the sense of \cite{R. Street}).
Using the notation in Lemma \ref{lem:lambda_fork}, in the diagram
\begin{equation}
\xymatrix@C=35pt{
(P\bullet I)\circ {L_\bullet}\circ A
\ar@<2pt>[r]^-{\gamma \circ A}
\ar@<-2pt>[r]_-{(P\bullet I)\circ\alpha}
& (P\bullet I)\circ A\ar[d]_-{\lambda_P^0}\ar[r]^-{\pi^L_{P\bullet I,A}}
& (P\bullet I)\circ_{L_\bullet}A \ar@{-->}@/1.5pc/[d]^-{\lambda_P}\\
& (P\circ A)\bullet_{R_\circ} A\ar[r]^-{\pi^L_{P,A}\bullet_{R_\circ} A}
&(P\circ_{{L_\bullet}} A)\bullet_{R_\circ} A}
\label{eq2.8}
\end{equation}
the top row is a coequalizer as in \eqref{eq:pi_L}. 
By Lemma \ref{lem:lambda_fork},
$(\pi^L_{P,A}\bullet_{R_\circ} A).\lambda_P^0$ coequalizes the parallel
morphisms in the top row.
So by universality, the (unique) morphism $\lambda_P$ exists.
Since $\lambda^0$ and $\pi^L$ are natural, so is $\lambda$ by its
construction in \eqref{eq2.8}.
We know from Proposition \ref{prop:rel_Hopf_lifting} that $\lambda_P^0$ is a
morphism of $(-)\circ A$-modules, hence so is
$(\pi^L_{P,A}\bullet_{R_\circ} A).\lambda_P^0$. Since \eqref{eq:chi_L} is a
morphism of $(-)\circ A$-modules (with respect to the action given by
multiplication in the last factor), so is $\pi^L_{P,A}$ for any
$L_\bullet$-module $P$. From these and \eqref{eq2.8} we conclude that
$\lambda_P$ is a morphism of $(-)\circ A$-modules.
The compatibilities of $\lambda$ with the comultiplications and the
counits of the comonads $(-)\bullet I$ on ${\mathcal M}_{{L_\bullet}}$ and
$(-)\bullet_{R_\circ} A$ on ${\mathcal M}_A$ follow by naturality of $\pi^L$ in
its first argument, functoriality of the product $\bullet_{R_\circ}$
and the compatibilities of $\lambda^0$, which are consequences of
Proposition \ref{prop:rel_Hopf_lifting}. 
\end{proof}

Corresponding (as in \cite[Proposition 1.1]{GomTor}) to the
comonad morphism $\lambda$ in Proposition \ref{prop:lambda}, there is a
`Galois-type' morphism of comonads 
\begin{equation}\label{eq:beta}
\xymatrix@C=40pt{
\beta_Q:(Q\bullet I)\circ_{L_\bullet} A\ar[r]^-{\lambda_Q}&
(Q \circ_{L_\bullet} A)\bullet_{R_\circ} A
\ar[r]^-{\overline \gamma\bullet_{R_\circ} A}&
Q \bullet_{R_\circ} A,}
\end{equation}
for any right $A$-module $(Q,\gamma)$, where $\overline \gamma$ is the unique
morphism for which $\overline \gamma.\pi^L_{Q,A}=\gamma$.
This is the unique morphism rendering commutative
$$
\xymatrix{
(Q\bullet I)\circ_{L_\bullet} A\ar@{-->}[rrr]^-{\beta_Q}&&&
Q \bullet_{R_\circ} A\ar[d]^-{\iota^R_{Q,A}}\\
(Q\bullet I)\circ A
\ar[u]^-{\pi^L_{Q\bullet I,A}}\ar[r]_-{(Q\bullet I)\circ \Delta}&
(Q\bullet I)\circ (A\bullet A)\ar[r]_-\zeta&
(Q \circ A)\bullet A\ar[r]_-{\gamma \bullet A}&
Q\bullet A.}
$$

\section{The Fundamental Theorem of Hopf modules}\label{sec:fthm}

By Proposition \ref{prop:lambda}, there is a functor $(-)\circ_{L_\bullet} A:
\mathcal M^I_{L_\bullet}\to \mathcal M^A_A$, for any weak bimonoid $A$ in a
duoidal category $(\mathcal M,\circ,\bullet)$ in which idempotent morphisms
split. The aim of this section is to investigate when this is an equivalence;
that is, when the fundamental theorem of Hopf modules holds. In the case of a
weak bialgebra $A$ over a field, it is known to be the case if and only if $A$
is a weak Hopf algebra (see \cite[Section 36.16]{BrzWis}). 

Let $(\mathcal M,\circ,\bullet)$ be a duoidal category in which idempotent
morphisms split and let $A$ be a bimonoid in it. In the diagram
$$
\xymatrix@C=45pt{
&&\mathcal M^A_A\ar[d]\\
\mathcal M^I_{L_\bullet}\ar[r]_-{V^I}\ar[rru]^-{(-)\circ_{L_\bullet} A}&
\mathcal M_{L_\bullet}\ar[r]_-{\omega^*=(-)\circ_{L_\bullet} A}&
\mathcal M_A,
}
$$
the vertical arrow denotes the comonadic forgetful functor and the functor in
the bottom row --- to be denoted by $Y$ --- is left adjoint. Hence the
diagonal functor $(-)\circ_{L_\bullet} A: \mathcal M^I_{L_\bullet}\to \mathcal
M^A_A$ is an equivalence if and only if \eqref{eq:beta} is a natural
isomorphism and $Y=(-)\circ_{L_\bullet} A: \mathcal M^I_{L_\bullet}\to
\mathcal M_A$ is comonadic (see e.g. \cite[Theorem 2.7]{GomTor} or
\cite[Theorem 1.10]{B. Mesablishvili}). By this motivation, applying methods
in \cite{B. Mesablishvili}, we turn to studying when $Y$ is comonadic.

Regard the Kleisli category $\widetilde{\mathcal M}_A$ of the monad
$\omega_\ast \omega^\ast=(-)\circ_{L_\bullet} A$ on $\mathcal
M_{L_\bullet}$ as a full subcategory (of free modules) in the Eilenberg-Moore
category of $\omega_\ast \omega^\ast$; which is isomorphic to $\mathcal
M_A$. The functor $Y=(-)\circ_{L_\bullet} A: \mathcal
M^I_{L_\bullet}\to \mathcal M_A$ factorizes through $\widetilde
Y:=(-)\circ_{L_\bullet} A: \mathcal M^I_{L_\bullet}\to \widetilde{\mathcal
M}_A$ via the fully faithful embedding $\widetilde{\mathcal M}_A\to
{\mathcal M}_A$. 

By a symmetric version of Lemma \ref{lem:R_bimod}, $L_\circ$ carries the
structure of an $A$-$J$ bimodule in $(\mathcal M,\circ)$. For any object
$P\circ_{L_\bullet} A$ of $\widetilde {\mathcal M}_A$ and any left $A$-module
$Q$, there exists the coequalizer $P\circ_{L_\bullet} A \circ_A Q\cong
P\circ_{L_\bullet} Q$, see e.g. \cite[Remark 2.4]{B. Mesablishvili}
(where $Q$ is a left $L_\bullet$-module via the action induced by the monoid
morphism $\omega=(A\bullet \tau).\iota^L_\bullet:L_\bullet \to A$, cf. a
symmetric version of Lemma \ref{lem:comonoid_map}). 
Thus for any right $L_\bullet$-module $P$, there is a right $J$-module
$P\circ_{L_\bullet} A \circ_A L_\circ \cong P\circ_{L_\bullet} L_\circ$. By
Proposition \ref{prop:bimodiso}, $L_\circ\cong L_\bullet \circ J$ as
$L_\bullet$-$J$ bimodules. Therefore
$$
P\circ_{L_\bullet} A \circ_A L_\circ \cong
P\circ_{L_\bullet} L_\circ\cong
P\circ_{L_\bullet} L_\bullet \circ J\cong
P \circ J,
$$
resulting in a commutative (up-to natural isomorphism) diagram
\begin{equation}\label{eq:HW}
\xymatrix@C=35pt{
\mathcal M^I_{L_\bullet} \ar[r]^-{V^I} \ar[d]_-{V_{L_\bullet}}
& \mathcal M_{L_\bullet}\ar[d]^-{U_{L_\bullet}}\ar[r]^-{(-)\circ_{L_\bullet} A}
& \widetilde{\mathcal M}_A \ar[d]^-{(-)\circ_A L_\circ}\\
\mathcal M^I\ar[r]_-{U^I}&\mathcal M\ar[r]_-{(-)\circ J}& \mathcal M_J.}
\end{equation}

\begin{proposition}\label{prop:MW}
Let $(\mathcal M,\circ,\bullet)$ be a duoidal category in which
idempotent morphisms split and the functor
$
\xymatrix{
\mathcal{M}^I \ar[r]^-{U^I}
& \mathcal{M}\ar[r]^-{(-)\circ J}
& \mathcal{M}_J}
$
in the bottom row of \eqref{eq:HW} is separable (in the sense of
\cite{NaVdBVO}). Then for any weak bimonoid $A$ in $\mathcal M$, the functor
$\widetilde Y=(-)\circ_{L_\bullet} A: \mathcal M^I_{L_\bullet}\to
\widetilde{\mathcal M}_A$ in the top row of \eqref{eq:HW} obeys the following
properties.
\begin{itemize}
\item[{(1)}] $\widetilde Y$ reflects isomorphisms.
\item[{(2)}] Any $\widetilde Y$-contractible pair of morphisms (in the dual
 sense of \cite[page 93, Section 3.3]{M. Barr}) possesses a contractible
 (hence absolute) equalizer in $\mathcal M^I_{L_\bullet}$.
\end{itemize}
\end{proposition}

\begin{proof}
By a symmetric version of Lemma \ref{lem:R_J-sFmonoid}, $L_\bullet$ is a
Frobenius monoid in $\mathcal M^I$. Applying the results about Frobenius
monads in \cite[Lemma 1.3 and Proposition 1.4]{Street:Frob}, both of
its category of modules, and its category of comodules are isomorphic to
$\mathcal M^I_{L_\bullet}$; and the forgetful functor $V_{L_\bullet}:\mathcal
M^I_{L_\bullet}\to \mathcal M^I$ possesses isomorphic left and right adjoints
$(-)\circ L_\bullet$. Moreover, since $L_\bullet$ is a separable comonoid in
$\mathcal M^I$, the forgetful functor $V_{L_\bullet}:\mathcal
M^I_{L_\bullet}\to \mathcal M^I$ is separable by Rafael's theorem
\cite{Rafael}: The unit of the adjunction $V_{L_\bullet} \dashv (-)\circ
L_\bullet$ is given by the $L_\bullet$-coaction $\varrho:X\to X \circ
L_\bullet$ at any object $X$ of $\mathcal M^I_{L_\bullet}$. It has a natural
retraction given by the $L_\bullet$-action 
$\xymatrix@C=15pt{
X\circ L_\bullet
\ar[r]^-{\varrho \circ L_\bullet}&
X\circ L_\bullet\circ L_\bullet
\ar[rr]^-{X \circ \varepsilon^L_\bullet.
\mu^L_\bullet}&&
X}$ on $X$.

The forgetful functor $U^I:\mathcal M^I \to \mathcal M$ has a right adjoint
$(-)\bullet I: \mathcal M \to \mathcal M^I$ and $(-)\circ J: \mathcal M \to
\mathcal M_J$ is left adjoint of the forgetful functor $U_J:\mathcal M_J \to
\mathcal M$. Thus the down-then-right path in \eqref{eq:HW} is a composite of
left adjoint separable functors, hence it is a separable and left adjoint
functor. Since idempotent morphisms are assumed to split in $\mathcal M$, they
also split in $\mathcal M^I_{L_\bullet}$. Thus the claims follow by
\cite[Proposition 1.13]{B. Mesablishvili}. (In fact, in \cite[Proposition
1.13]{B. Mesablishvili} part (2) of our proposition appears in a slightly
different form. There the claim is stated about {\em coreflexive} $\widetilde
Y$-contractible {\em equalizer} pairs; though these additional assumptions are
not used in their proof.) 
\end{proof}

With these preparations, the following `fundamental theorem of Hopf modules'
is obtained.

\begin{theorem}\label{thm:fthm}
Let $(\mathcal M,\circ,\bullet)$ be a duoidal category in which idempotent
morphisms split and the functor
$\xymatrix@C=15pt{
\mathcal{M}^I \ar[r]^-{U^I}
& \mathcal{M}\ar[r]^-{(-)\circ J}
& \mathcal{M}_J}$
is separable. Let $A$ be a weak bimonoid in $\mathcal M$. The corresponding
Galois morphism (\ref{eq:beta}) is an isomorphism if and only if the functor
$(-)\circ_{L_\bullet} A:\mathcal M^I_{L_\bullet} \to \mathcal M^A_A$ is an
equivalence. 
\end{theorem}

\begin{proof}
We use the reasoning applied in \cite[Proposition 3.7]{B. Mesablishvili} in
the non-weak case, to show that $Y=(-)\circ_{L_\bullet} A: \mathcal
M^I_{L_\bullet}\to {\mathcal M}_A$ is comonadic.

Since $\widetilde Y=(-)\circ_{L_\bullet} A: \mathcal M^I_{L_\bullet}\to
\widetilde{\mathcal M}_A$ reflects isomorphisms by Proposition \ref{prop:MW},
clearly so does $Y=(-)\circ_{L_\bullet} A: \mathcal M^I_{L_\bullet}\to
{\mathcal M}_A$ differing from $\widetilde Y$ by the fully faithful embedding
$\widetilde{\mathcal M}_A \to {\mathcal M}_A$.

Since idempotent morphisms in $\mathcal M$ split, they also split in
$\mathcal M_A$. So in this case a $Y$-contractible equalizer pair (in the dual
sense of \cite[page 94]{M. Barr}) is simply a $Y$-contractible pair (in the
dual sense of \cite[page 93, Section 3.3]{M. Barr}). Using again that $Y$ and
$\widetilde Y$ differ by the fully faithful embedding $\widetilde{\mathcal
 M}_A \to {\mathcal M}_A$, we conclude that a $Y$-contractible pair is the
same as a $\widetilde Y$-contractible pair. By Proposition \ref{prop:MW},
$\widetilde Y$-contractible pairs (i.e. $Y$-contractible equalizer pairs)
possess absolute equalizers. This proves that $Y$ creates equalizers of
$Y$-contractible equalizer pairs.

Since there are adjunctions
$$
\xymatrix@C=40pt{
\mathcal M^I_{L_\bullet}\ar@/^1pc/[r]^-{V^I}\ar@{}[r]|-\perp&
\mathcal M_{L_\bullet}\ar@/^1pc/[l]^-{(-)\bullet I}
\ar@/^1pc/[r]^-{\omega^*=(-)\circ_{_{L_\bullet}} A} \ar@{}[r]|-\perp&
\mathcal M_A,\ar@/^1pc/[l]^-{\omega_*}}
$$
$Y$ (occurring in the top row) possesses a right adjoint. So it follows by
the dual form of Beck's theorem \cite{J. Beck} (see \cite[page 100, Theorem
3.14]{M. Barr}) that $Y$ is comonadic. Since $(-)\circ_{L_\bullet} A: \mathcal
M^I_{L_\bullet}\to \mathcal M^A_A$ is an equivalence if and only if
\eqref{eq:beta} is a natural isomorphism and $Y$ is comonadic (see
e.g. \cite[Theorem 2.7]{GomTor} or \cite[Theorem 1.10]{B. Mesablishvili}),
this completes the proof. 
\end{proof}

Applying all possible composites of the duality transformations $\bullet$,
$\circ$ and $\ast$ in Section \ref{sec:axioms}, we obtain seven further
symmetric variants of the equivalent conditions in Theorem \ref{thm:fthm}. In
a braided monoidal category --- regarded as a duoidal category --- whose
idempotent morphisms split, half of them reduce to the usual definition of
weak Hopf monoid in \cite{C. Pastro,AlAlatal}; and half of them reduce to the
condition that the opposite weak bimonoid is a weak Hopf monoid. 
\vfill
\eject

\begin{landscape}
\textwidth33cm
\pagestyle{plain}

\section*{Appendix: Some diagrammatic proofs}
\bigskip

{\white .}\hspace{-3cm} {\em Proof of Lemma \ref{lem:kappa} (3).}
\bigskip

\scalebox{.95}{
$$\label{page_lem:kappa}
{\white .}\hspace{-3.6cm}\xymatrix@C=10pt{
((J\circ A)\bullet X)\circ (A\bullet Y)
\ar[r]^-{\raisebox{7pt}{${}_
{((J\circ A\circ \Delta.\eta)\bullet X)\circ (A\bullet Y)}$}}
\ar[ddddd]^-\zeta&
((J\circ A\circ (A\bullet A))\bullet X)\circ (A\bullet Y)
\ar[dd]^-{((J\circ \zeta)\bullet X)\circ (A\bullet Y)}
\ar@{=}[rrr]
\ar[rd]^-{\qquad ((\zeta \circ (A\bullet A))\bullet X)\circ (A\bullet Y)}&&&
((J\circ A\circ (A\bullet A))\bullet X)\circ (A\bullet Y)
\ar[dd]_-{(\zeta\bullet X)\circ (A\bullet Y)}\\
&&
((((J\circ A)\bullet (J\circ J))\circ (A\bullet A))\bullet X)
\circ (A\bullet Y)
\ar[rru]^-{((((J\circ A)\bullet \varpi)\circ (A\bullet A))\bullet X)
\circ (A\bullet Y)\qquad \ \ }
\ar[d]^-{(\zeta\bullet X)\circ (A\bullet Y)}
\ar@{}[u]|-{\eqref{eq1.2}}\\
&
((J\circ ((A\circ A)\bullet (J\circ A)))\bullet X)\circ (A\bullet Y)
\ar[r]^-{\raisebox{7pt}{${}_{(\zeta\bullet X)\circ (A\bullet Y)}$}}
\ar[d]^-{((J\circ (\varepsilon.\mu\bullet (J\circ A)))\bullet X)
\circ (A\bullet Y)}
\ar@{}[ruu]|(.4){\eqref{eq1.1}}&
((J\circ A\circ A)\bullet (J\circ J \circ A)\bullet X)\circ (A\bullet Y)
\ar[d]^-{((J\circ \varepsilon.\mu)\bullet (J\circ J \circ A)\bullet X)
\circ (A\bullet Y)}&&
((J\circ A\circ A)\bullet (J\circ A)\bullet X)\circ (A\bullet Y)
\ar[d]_-{((J\circ \varepsilon.\mu)\bullet (J\circ A)\bullet X)
\circ (A\bullet Y)}\\
&
((J\circ J \circ A)\bullet X)\circ (A\bullet Y)
\ar[r]^-{(\zeta\bullet X)\circ (A\bullet Y)}\ar@{=}[rd]&
((J\circ J)\bullet (J\circ J \circ A)\bullet X)\circ (A\bullet Y)
\ar[d]^-{(\varpi\bullet (J\circ J \circ A)\bullet X)\circ (A\bullet Y)}
\ar@{}[ld]|(.3){\eqref{eq1.2}}&&
((J\circ J)\bullet (J\circ A)\bullet X)\circ (A\bullet Y)
\ar[d]_-{(\varpi \bullet (J\circ A)\bullet X)\circ (A\bullet Y)}\\
&&
((J\circ J \circ A)\bullet X)\circ (A\bullet Y)
\ar[rr]^-{((\varpi \circ A)\bullet X)\circ (A\bullet Y)}&&
((J\circ A)\bullet X)\circ (A\bullet Y)\ar[d]_-\zeta\\
(J\circ A \circ A)\bullet (X\circ Y)
\ar[r]^-{\raisebox{7pt}{${}_
{(J\circ A \circ \Delta.\eta\circ A)\bullet (X\circ Y)}$}}
\ar[dd]^-{(J\circ \mu)\bullet (X\circ Y)}
\ar@{}[rrrdd]|-{(\LRC)}&
(J\circ A \circ (A \bullet A) \circ A)\bullet (X\circ Y)
\ar[r]^-{(J\circ \zeta \circ A)\bullet (X\circ Y)}&
(J\circ ((A\circ A)\bullet (J\circ A))\circ A)\bullet (X\circ Y)
\ar[r]^-{\raisebox{7pt}{${}_
{(J\circ (\varepsilon.\mu\bullet (J\circ A))\circ A)\bullet (X\circ Y)}$}}&
(J\circ J\circ A \circ A)\bullet (X\circ Y)
\ar[r]^-{\raisebox{7pt}{${}_{(\varpi\circ A \circ A)\bullet (X\circ Y)}$}}
\ar[d]^-{(J\circ J\circ \varepsilon.\mu)\bullet (X\circ Y)}&
(J\circ A \circ A)\bullet (X\circ Y)
\ar[d]_-{(J\circ \varepsilon.\mu)\bullet (X\circ Y)}\\
&&&
(J\circ J \circ J)\bullet (X\circ Y)
\ar[r]^-{(\varpi \circ J)\bullet (X\circ Y)}
\ar[d]^-{(J\circ \varpi)\bullet (X\circ Y)}&
(J\circ J)\bullet (X\circ Y)
\ar[d]_-{\varpi\bullet (X\circ Y)}\\
(J\circ A)\bullet (X\circ Y)
\ar[rrr]_-{(J\circ \varepsilon)\bullet (X\circ Y)}&&&
(J\circ J)\bullet (X\circ Y)\ar[r]_-{\varpi\bullet (X\circ Y)}&
X\circ Y
}
$$}

\eject

{\white .}\hspace{-3cm} {\em Proof of Theorem \ref{thm:wbm}, verification of
 \cite[eq. (1.5)]{G. Bohm2011}.}
\vspace{1cm}

\scalebox{.85}{
$$\label{page_1.5}
{\white .}\hspace{-4cm}
\xymatrix@C=10pt@R=40pt{
X\bullet ((Y\bullet Z)\circ (A\bullet A))
\ar[rr]^-{X\bullet \zeta}&&
X\bullet (Y\circ A)\bullet (Z\circ A)
\ar[r]^-{\raisebox{7pt}{${}_{
((X\bullet (Y\circ A))\circ \Delta.\eta)\bullet(Z\circ A)}$}}&
((X\bullet (Y\circ A))\circ (A\bullet A))\bullet(Z\circ A)
\ar[rr]^-{\zeta \bullet (Z\circ A)}&&
(X\circ A)\bullet (Y\circ A\circ A)\bullet (Z\circ A)
\ar[dddd]_-{(X\circ A)\bullet (Y\circ \mu)\bullet (Z\circ A)}\\
&
(X\bullet Y \bullet Z)\circ (I\bullet A \bullet A)\ar[r]^-\zeta
\ar[lu]_-\zeta\ar@{=}[d]\ar@{}[ru]|-{\eqref{eq1.1}}&
((X\bullet Y)\circ (I\bullet A))\bullet (Z\circ A)
\ar[u]_-{\zeta \bullet (Z\circ A)}
\ar[r]^-{\raisebox{7pt}{${}_{
((X\bullet Y)\circ (I\bullet A)\circ \Delta.\eta)\bullet(Z\circ A)}$}}&
((X\bullet Y)\circ (I\bullet A)\circ (A\bullet A))\bullet (Z\circ A)
\ar[u]_-{(\zeta \circ (A\bullet A))\bullet (Z\circ A)}
\ar[r]^-{\raisebox{7pt}{${}_{
((X\bullet Y)\circ \zeta)\circ (Z\bullet A)}$}}\ar@{}[rru]|-{\eqref{eq1.1}}&
((X\bullet Y)\circ (A\bullet (A\circ A)))\bullet (Z\circ A)
\ar[ru]_-{\zeta \bullet (Z\circ A)}
\ar[ddd]_-{((X\bullet Y)\circ (A\bullet \mu))\bullet (Z\circ A)}\\
X\bullet Y \bullet Z\ar@{=}[dd]^-{\qquad \eqref{eq1.2}}
\ar[uu]_-{X\bullet ((Y\bullet Z)\circ \Delta.\eta)}&
(X\bullet Y \bullet Z)\circ (I\bullet A \bullet A)
\ar[r]^-{\raisebox{7pt}{${}_{(X\bullet Y \bullet Z)\circ (((I\bullet A)
\circ \Delta.\eta)\bullet A)}$}}\ar@{}[rrdd]|-{(\RRU)}&
(X\bullet Y \bullet Z)\circ(((I\bullet A)\circ (A\bullet A))\bullet A)
\ar[r]^-{\raisebox{7pt}{${}_{(X\bullet Y \bullet Z)\circ(\zeta\bullet A)}$}}&
(X\bullet Y \bullet Z)\circ (A\bullet (A\circ A)\bullet A)
\ar[d]_-{(X\bullet Y \bullet Z)\circ (A\bullet \mu \bullet A)}\\
&
(X\bullet Y \bullet Z)\circ (I\bullet I)
\ar[u]_-{(X\bullet Y \bullet Z)\circ (I\bullet \Delta.\eta)}
\ar[lu]_-\zeta
&&
(X\bullet Y \bullet Z)\circ (A \bullet A \bullet A)
\ar[dr]^-\zeta
\\
X\bullet Y \bullet Z
\ar[r]_-{(X\bullet Y \bullet Z)\circ \eta}
\ar[ru]_-{(X\bullet Y \bullet Z)\circ \delta}&
(X\bullet Y \bullet Z)\circ A
\ar[r]_-{(X\bullet Y \bullet Z)\circ \Delta}
&
(X\bullet Y \bullet Z)\circ (A\bullet A)
\ar[ru]^-{(X\bullet Y \bullet Z)\circ (\Delta\bullet A)\qquad }
\ar[r]_-\zeta&
((X\bullet Y)\circ A)\bullet (Z\circ A)
\ar[r]_-{((X\bullet Y)\circ \Delta)\bullet (Z\circ A)}&
((X\bullet Y) \circ (A\bullet A))\bullet (Z\circ A)
\ar[r]_-{\zeta \bullet (Z\circ A)}&
(X\circ A)\bullet (Y\circ A)\bullet (Z\circ A)}
$$}
\eject

{\white .}\hspace{-3cm} {\em Proof of Theorem \ref{thm:wbm}, verification of
 \cite[eq. (1.4)]{G. Bohm2011}.}
\vspace{1cm}

\scalebox{.85}{
$$\label{page_1.4}
{\white .}\hspace{-4.3cm}
\xymatrix@C=10pt{
((J\circ A)\bullet X)\circ A
\ar[r]^-{((J\circ A)\bullet X)\circ \Delta}
\ar@{=}[ddd]&
((J\circ A)\bullet X)\circ (A\bullet A)
\ar[r]^-\zeta
\ar@{=}[rd]_(.3){\eqref{eq1.2}\quad}
\ar@/_2pc/[rrrr]_-{\kappa_{X,A}}
\ar[d]_-{((J\circ A)\bullet X)\circ \delta \circ (A\bullet A)}&
(J\circ A\circ A)\bullet (X\circ A)
\ar[rr]^-{(J\circ \varepsilon.\mu)\bullet (X\circ A)}&&
(J\circ J)\bullet (X\circ A)
\ar[r]^-{\varpi \bullet (X\circ A)}&
X\circ A\\
&
((J\circ A)\bullet X)\circ (I\bullet I)\circ (A\bullet A)
\ar[r]^-{\zeta \circ (A\bullet A)}
\ar[d]_-{((J\circ A)\bullet X)\circ (\Delta.\eta\bullet I)\circ (A\bullet A)}&
((J\circ A)\bullet X)\circ (A\bullet A)
\ar[d]^-{((J\circ A\circ \Delta.\eta)\bullet X)\circ (A\bullet A)}&&&
(J\circ J)\bullet (X\circ A)
\ar[u]^-{\varpi\bullet (X\circ A)}\\
&
((J\circ A)\bullet X)\circ (A\bullet A\bullet I)\circ (A\bullet A)
\ar[r]^-{\zeta \circ (A\bullet A)}
\ar[rrrddd]^-{\kappa_{X,A\bullet I}\circ (A\bullet A)}&
((J\circ A\circ (A\bullet A))\bullet X)\circ (A\bullet A)
\ar[rr]^-{(\kappa_{J,A}\bullet X)\circ (A\bullet A)}&
\ar@{}[uu]|(.3){\textrm{Lemma\ \ref{lem:kappa}\ } (3)}
\ar@{}[dd]|-{\textrm{Lemma\ \ref{lem:kappa}\ } (2)}&
((J\circ A)\bullet X)\circ (A\bullet A)
\ar[r]^-\zeta
\ar@/^1pc/[ruu]^-{\kappa_{X,A}}
\ar@{}[dddr]|-{\eqref{eq1.1}}&
(J\circ A\circ A)\bullet (X\circ A)
\ar[u]^-{(J\circ \varepsilon.\mu)\bullet (X\circ A)}\\
((J\circ A)\bullet X)\circ A
\ar[r]^-{((J\circ A)\bullet X)\circ \delta \circ A}
\ar[ddddd]^-{((J\circ A)\bullet X)\circ \Delta.\eta\circ A}
\ar@{}[rdddd]|-{(\LRU)}&
((J\circ A)\bullet X)\circ (I\bullet I)\circ A
\ar[d]^-{((J\circ A)\bullet X)\circ (\Delta.\eta\bullet I)\circ A}
&&\\
&
((J\circ A)\bullet X)\circ (A\bullet A\bullet I)\circ A
\ar[r]^-{\kappa_{X,A\bullet I}\circ A}
\ar[dd]^-{((J\circ A)\bullet X)\circ
(A\bullet ((A\bullet I)\circ \Delta.\eta))\circ A}
\ar@{}[rddd]|-{\textrm{Lemma\ \ref{lem:kappa}\ } (1)}&
X\circ (A\bullet I)\circ A
\ar[d]^-{X\circ (A\bullet I)\circ \eta\circ A}
\ar@{=}[rd]
&&\\
&&
X\circ (A\bullet I)\circ A \circ A
\ar[r]^-{X\circ (A\bullet I)\circ \mu}
\ar[d]^-{X\circ (A\bullet I)\circ \Delta \circ A}
\ar@{}[rrd]|-{\eqref{eq:WB}}&
X\circ (A\bullet I)\circ A
\ar[r]^-{X\circ (A\bullet I)\circ \Delta}&
X\circ (A\bullet I)\circ (A\bullet A)
\ar[r]^-{X\circ \zeta}
\ar[uuu]^-{\zeta\circ (A\bullet A)}&
X\circ ((A\circ A)\bullet A)
\ar@{=}[dd]
\ar[uuu]_(.4)\zeta\\
&
((J\circ A)\bullet X)\circ (A\bullet ((A\bullet I)\circ (A\bullet A)))\circ A
\ar[d]^-{((J\circ A)\bullet X)\circ (A\bullet \zeta)\circ A}&
X\circ (A\bullet I)\circ (A\bullet A) \circ A
\ar[r]^-{\raisebox{7pt}{${}_
{X\circ (A\bullet I)\circ (A\bullet A) \circ\Delta}$}}
\ar[d]^-{X\circ \zeta \circ A}&
X\circ (A\bullet I)\circ (A\bullet A) \circ (A\bullet A)
\ar[r]^-{\raisebox{7pt}{${}_
{X\circ (A\bullet I)\circ\zeta}$}}
\ar[d]^-{X\circ \zeta \circ (A\bullet A)}
\ar@{}[rd]|-{\eqref{eq1.1}}&
X \circ (A\bullet I)\circ ((A\circ A)\bullet (A\circ A))
\ar[d]_-{X\circ \zeta}
\ar[u]^-{X \circ (A\bullet I)\circ (\mu \bullet \mu)}\\
&
((J\circ A)\bullet X)\circ (A\bullet (A\circ A)\bullet A) \circ A
\ar[d]^-{((J\circ A)\bullet X)\circ (A\bullet \mu \bullet A) \circ A}&
X\circ ((A\circ A)\bullet A)\circ A
\ar[d]^-{X\circ (\mu \bullet A)\circ A}&
X\circ ((A\circ A)\bullet A)\circ (A\bullet A)
\ar[r]^-{X\circ \zeta}
\ar[d]^-{X\circ (\mu \bullet A) \circ (A\bullet A)}&
X\circ ((A\circ A\circ A)\bullet (A\circ A))
\ar[r]^-{X\circ ((A\circ \mu)\bullet \mu)}
\ar[d]_-{X\circ ((\mu\circ A)\bullet (A\circ A))}&
X\circ ((A\circ A)\bullet A)
\ar[d]_-{X\circ (\mu \bullet A)}\\
((J\circ A)\bullet X)\circ (A\bullet A)\circ A
\ar[r]^-{\raisebox{7pt}{${}_
{((J\circ A)\bullet X)\circ (A\bullet \Delta)\circ A}$}}
\ar[d]^-{\zeta \circ A}
\ar[rrd]^-{\kappa_{X,A}\circ A}&
((J\circ A)\bullet X)\circ (A\bullet A\bullet A)\circ A
\ar[r]^-{\kappa_{X,A\bullet A}\circ A}
\ar@{}[rd]^-{\textrm{Lemma\ \ref{lem:kappa}\ } (1)}&
X\circ (A\bullet A)\circ A
\ar[r]^-{X\circ (A\bullet A)\circ \Delta}
\ar@{}[rrrd]|-{\eqref{eq:WB}}&
X\circ (A\bullet A)\circ (A\bullet A)
\ar[r]^-{X\circ \zeta}&
X\circ ((A\circ A)\bullet (A\circ A))
\ar[r]^-{X\circ (\mu \bullet \mu)}&
X\circ (A\bullet A)
\ar[d]_-{X\circ (\varepsilon \bullet A)}\\
((J\circ A\circ A)\bullet (X\circ A))\circ A
\ar[r]_-{
\raisebox{-7pt}{${}_{((J\circ \varepsilon.\mu)\bullet (X\circ A))\circ A}$}}&
((J\circ J)\bullet (X\circ A))\circ A
\ar[r]_-{(\varpi \bullet (X\circ A))\circ A}&
X\circ A \circ A
\ar[rr]_-{X\circ \mu}
\ar[u]^-{X\circ \Delta \circ A}&&
X\circ A
\ar[ru]^-{X\circ \Delta}
\ar@{=}[r]&
X\circ A
\ar@/_5pc/[uuuuuuuu]^(.65)\zeta
}
$$}
\eject

{\white .}\hspace{-2cm}
{\em Proof of Lemma \ref{lem:pi_colinear}, computation of
$((J\circ A)\bullet \sqcap^R_\circ) .\zeta.(J\circ\Delta).\sqcap^R_\circ$.}
\vspace{1cm}

\scalebox{1}{
$$\label{page_sqcap_colinear_1}
{\white .}\hspace{-2cm}
\xymatrix@C=15pt@R=30pt{
J\circ A
\ar@{=}[dd]
\ar[rr]^{J\circ A\circ\eta} &&
J\circ A\circ A
\ar@{}[rrdd]|-{(\ast)}
\ar[rr]^{J\circ A\circ\Delta} &&
J\circ A\circ(A\bullet A)
\ar[r]^{J\circ A\circ(\Delta\bullet A)} &
J\circ A\circ(A\bullet A\bullet A)
\ar[d]_-\zeta\\
&&&&&(J\circ (A\bullet A))\bullet (J\circ A \circ A)
\ar[d]_-{\zeta\bullet (J\circ A \circ A)}\\
J\circ A
\ar[r]^-{\raisebox{6pt}{${}_{J\circ A \circ \Delta.\eta}$}}
\ar@{=}[dd]
\ar@{}[rrdd]|-{\eqref{eq:sqcap^R}}&
J\circ A \circ (A\bullet A)
\ar[r]^-\zeta&
(J\circ A)\bullet (J\circ A \circ A)
\ar[d]^-{(J\circ A)\bullet (J\circ \varepsilon.\mu)}
\ar[r]^-{\raisebox{6pt}
{${}_{(J\circ A\circ \Delta.\eta)\bullet (J\circ A \circ A)}$}}&
(J\circ A\circ (A\bullet A))\bullet (J\circ A \circ A)
\ar[r]^-{\raisebox{6pt}{${}_{\zeta\bullet (J\circ A \circ A)}$}}&
(J\circ A)\bullet (J\circ A\circ A)\bullet (J\circ A \circ A)
\ar[r]^-{\raisebox{6pt}{
${}_{(J\circ A)\bullet (J\circ \mu)\bullet (J\circ A \circ A)}$}}&
(J\circ A)\bullet (J\circ A)\bullet (J\circ A \circ A)
\ar[d]_-{(J\circ A)\bullet (J\circ A)\bullet (J\circ \varepsilon.\mu)}\\
&&
(J\circ A)\bullet (J\circ J)
\ar[d]^-{(J\circ A)\bullet \varpi}&&&
(J\circ A)\bullet (J\circ A)\bullet (J\circ J)
\ar[d]_-{(J\circ A)\bullet (J\circ A)\bullet \varpi}\\
J\circ A
\ar@{=}[d]
\ar[rr]^-{\sqcap^R_\circ}&&
J\circ A
\ar[r]^-{J\circ A \circ \Delta.\eta}&
J\circ A \circ (A\bullet A)
\ar[r]^-\zeta&
(J\circ A)\bullet (J\circ A \circ A)
\ar[r]^-{(J\circ A)\bullet (J\circ \mu)}
\ar@{}[d]|-{\eqref{eq:wbm_1.10}}&
(J\circ A)\bullet (J\circ A)
\ar[d]_-{(J\circ A)\bullet \sqcap^R_\circ}\\
J\circ A
\ar[rr]_-{J\circ A \circ \Delta.\eta}&&
J\circ A \circ (A\bullet A)
\ar[r]_-\zeta&
(J\circ A)\bullet (J\circ A \circ A)
\ar[r]_-{(J\circ A)\bullet (J\circ \mu)}
\ar[ru]^-{(J\circ A)\bullet (\sqcap^R_\circ \circ A)\qquad}&
(J\circ A)\bullet (J\circ A)
\ar[r]_-{(J\circ A)\bullet \sqcap^R_\circ}&
(J\circ A)\bullet (J\circ A)}
$$}
\eject

{\white .}\hspace{-3cm} {\em Proof of Lemma \ref{lem:pi_colinear},
 computation of $(\sqcap^R_\circ\bullet(J\circ A)) .\zeta.(J\circ\Delta)$.}
\vspace{1cm}

\scalebox{.93}{
$$\label{page_sqcap_colinear_2}
{\white .}\hspace{-3.8cm}
\xymatrix@C=10pt @R=30pt{
J\circ A
\ar[rrrr]^-{J\circ \Delta}
\ar@{=}[dd]
\ar@{}[rrrrd]|-{\eqref{eq:WB}}&&&&
J\circ (A \bullet A)
\ar[r]^-\zeta
\ar@{=}[d]&
(J \circ A)\bullet (J\circ A)
\ar[dd]_-{(J \circ A\circ \Delta.\eta)\bullet (J\circ A)}\\
&
J\circ A \circ A
\ar[ul]_-{J\circ \mu}
\ar[r]^-{J\circ \Delta\circ \Delta}&
J\circ (A \bullet A) \circ (A \bullet A)
\ar[r]^-{J\circ \zeta}
\ar@{}[dd]|-{(\ast)}&
J\circ ((A\circ A) \bullet (A\circ A))
\ar[r]^-{J\circ (\mu\bullet \mu)}
\ar[d]^-{J\circ ((A\circ A\circ \Delta.\eta) \bullet (A\circ A))}&
J\circ (A \bullet A)\\
J\circ A
\ar[ur]_(.6){\ J\circ A \circ \eta}
\ar[dd]^-{J\circ A \circ \Delta.\eta}
\ar[r]^-{J\circ \Delta}&
J \circ (A\bullet A)
\ar[dd]^-{J \circ (A\bullet A) \circ \Delta^2.\eta}
\ar[ur]_(.6){\quad J \circ (A\bullet A) \circ \Delta.\eta}&
&
J\circ ((A\circ A \circ (A\bullet A)) \bullet (A\circ A))
\ar[r]^-{\raisebox{6pt}{${}_{J\circ ((\mu \circ (A\bullet A)) \bullet \mu)}$}}
\ar[d]^-{J\circ (\zeta \bullet (A\circ A))}&
J\circ ((A\circ (A\bullet A))\bullet A)
\ar[r]^-\zeta&
(J\circ A\circ (A\bullet A))\bullet (J\circ A)
\ar[d]_-{(J\circ \zeta)\bullet (J\circ A)}\\
&&
J\circ ((A\circ (A\bullet A))\bullet (A\circ A))
\ar[rd]^-{J\circ (\zeta \bullet (A\circ A))}
\ar@{}[d]|-{\eqref{eq1.1}}&
J\circ ((J\circ A)\bullet (A\circ A \circ A)\bullet (A\circ A))
\ar[r]^-{\raisebox{6pt}{
${}_{J\circ ((J\circ A)\bullet (\mu \circ A)\bullet \mu)}$}}
\ar[d]^-{J\circ ((J\circ A)\bullet (A\circ \mu)\bullet (A\circ A))}&
J\circ ((J\circ A)\bullet (A\circ A)\bullet A)
\ar[r]^-\zeta
\ar[d]^(.4){J\circ ((J\circ A)\bullet \varepsilon.\mu \bullet A)}&
(J\circ ((J\circ A)\bullet (A\circ A)))\bullet (J\circ A)
\ar[d]_(.6){(J\circ ((J\circ A)\bullet \varepsilon.\mu))\bullet (J\circ A)}\\
J\circ A \circ (A\bullet A)
\ar[r]^-{\raisebox{6pt}{${}_{J\circ \Delta \circ (A\bullet \Delta)}$}}
\ar@{=}[d]&
J\circ (A\bullet A) \circ (A\bullet A \bullet A)
\ar[r]^-{J\circ \zeta}
\ar[ru]^-{J\circ \zeta}&
J\circ ((J\circ A)\bullet ((A\bullet A)\circ (A\bullet A)))
\ar[r]^-{\raisebox{6pt}{${}_{J\circ ((J\circ A)\bullet \zeta)}$}}
\ar@{}[rrd]|-{\eqref{eq:WB}}&
J\circ ((J\circ A)\bullet (A\circ A)\bullet (A\circ A))
\ar[r]^-{\raisebox{6pt}{${}_{J\circ ((J\circ A)\bullet \varepsilon.\mu\bullet \mu)}$}}&
J\circ ((J\circ A)\bullet A)
\ar[r]^-\zeta
\ar@{=}[d]&
(J\circ J\circ A)\bullet (J\circ A)
\ar[d]_-{(\varpi \circ A)\bullet (J\circ A)}\\
J\circ A \circ (A\bullet A)
\ar[rr]^-{J\circ \zeta}
\ar[rd]^-{\zeta \circ (A\bullet A)}
\ar@{=}[dd]^-{\eqref{eq1.2}}
\ar@{}[rrd]|-{\eqref{eq1.1}}&&
J\circ ((J\circ A)\bullet (A\circ A))
\ar[u]_-{J\circ ((J\circ A)\bullet (\Delta\circ \Delta))}
\ar[rr]^-{J\circ ((J\circ A)\bullet \mu)}
\ar[d]^-\zeta&&
J\circ ((J\circ A)\bullet A)&
(J\circ A) \bullet (J\circ A)
\ar@{=}[dd]\\
&
((J\circ J)\bullet (J\circ A))\circ (A\bullet A)
\ar[r]^-\zeta
\ar[ld]^-{\quad (\varpi \bullet (J\circ A))\circ (A\bullet A)}&
(J\circ J \circ A)\bullet (J\circ A\circ A)
\ar[rr]^-{(\varpi \circ A)\bullet (J\circ A\circ A)}&&
(J \circ A)\bullet (J\circ A\circ A)
\ar@{=}[d]\\
J\circ A \circ (A\bullet A)
\ar[rrrr]_-\zeta&&&&
(J\circ A)\bullet (J\circ A \circ A)
\ar[r]_-{(J\circ A)\bullet (J\circ \mu)}&
(J \circ A)\bullet (J\circ A)}
$$}
\eject

{\white .}\hspace{-3cm} {\em Proof of Lemma \ref{lem:mu_colinear}.}
\vspace{1cm}

\scalebox{.99}{
$$\label{page_mu_colinear}
{\white .}\hspace{-2cm}
\xymatrix@C=20pt{
A \circ A \ar@{=}[rr]\ar[dd]^-{A\circ \Delta}&&
A\circ A \ar[rrr]^-{\mu}\ar[d]^-{\Delta\circ \Delta}
\ar@{}[rrrd]|-{\eqref{eq:WB}}&&&
A \ar[d]_-\Delta\\
&&
(A\bullet A)\circ (A\bullet A) \ar[r]^-\zeta 
\ar[d]^-{(A\bullet A)\circ (\Delta \bullet A)}&
(A\circ A)\bullet (A\circ A)\ar@{=}[r]
\ar[d]^-{(A\circ \Delta)\bullet (A \circ A)}
\ar@{}[rddd]|-{(A)}&
(A\circ A)\bullet (A\circ A)\ar[r]^-{\mu \bullet \mu}
\ar[d]^-{(A\circ A) \bullet \mu}&
A\bullet A\ar[dd]_-{(\tau\circ A)\bullet A}\\
A \circ (A\bullet A)
\ar[rr]^-{\Delta \circ (A\bullet \Delta)} \ar@{=}[d]&&
(A\bullet A)\circ (A\bullet A\bullet A)\ar[r]^-\zeta \ar[d]^-\zeta
\ar@{}[rd]|-{\eqref{eq1.1}}&
(A\circ (A\bullet A))\bullet (A\circ A) \ar[d]^-{\zeta \bullet (A\circ A)}&
(A\circ A)\bullet A\ar[d]^-{(\tau \circ A \circ A)\bullet A}\\
A \circ (A\bullet A)
\ar[r]^-{\raisebox{6pt}{${}_\zeta$}}
\ar[dd]^-{A\circ ((\tau\circ A)\bullet A)}&
(J\circ A)\bullet (A\circ A)
\ar[r]^-{\raisebox{6pt}{${}_{(J\circ A)\bullet (\Delta \circ \Delta)}$}}
\ar[dd]^-{(J\circ \tau\circ A)\bullet (A\circ A)}
\ar[rrdd]^-{(J\circ A)\bullet \mu}&
(J\circ A)\bullet ((A\bullet A)\circ (A\bullet A))
\ar[r]^-{\raisebox{6pt}{${}_{(J\circ A)\bullet \zeta}$}}
\ar@{}[rd]|-{\eqref{eq:WB}}&
(J\circ A)\bullet (A\circ A)\bullet (A\circ A)
\ar[d]^-{(J\circ A)\bullet \varepsilon.\mu\bullet (A\circ A)}&
(J\circ A \circ A) \bullet A \ar[r]^-{(J\circ\mu)\bullet A} 
\ar[dd]^-{\vartheta^R\bullet A}\ar@{}[rdd]|-{(B)}&
(J\circ A)\bullet A\ar[ddd]_-{\pi^R_\circ \bullet A}\\
&&&
(J\circ A)\bullet (A\circ A)\ar[d]^-{(J\circ A)\bullet \mu}&\\
A \circ ((J\circ A)\bullet A) \ar[d]^-{A \circ (\pi^R_\circ \bullet A)}&
(J\circ J \circ A)\bullet (A\circ A)\ar[rr]^-{(\varpi\circ A)\bullet \mu}
\ar[d]^-{(J\circ\pi^R_\circ)\bullet (A\circ A)}
\ar@{}[rrrd]|-{(C)}&&
(J\circ A)\bullet A\ar@{=}[r]&
(J\circ A)\bullet A\ar[rd]^-{\pi^R_\circ\bullet A}&\\
A \circ (R_\circ\bullet A)\ar[r]_-\zeta&
(J\circ R_\circ)\bullet (A\circ A)\ar[rr]_-{\gamma \bullet (A\circ A)}&&
 R_\circ \bullet (A\circ A)\ar[rr]_-{R_\circ\bullet \mu}&&
R_\circ\bullet A}
$$}

\eject
{\white .}\hspace{-3cm} {\em Proof of Proposition \ref{prop:Hopf_d_law},
 verification of the compatibility of $\psi$ with the unit.}
\vspace{1cm}

\scalebox{.84}{
$$\label{page_wdl_unit}
{\white .}\hspace{-4.3cm}
\xymatrix@C=10pt{
M\bullet A
\ar@{=}[rr]
\ar[dddddd]^-{(M\bullet A)\circ \eta}
\ar[rd]^-{(M\bullet A)\circ \delta}&
\ar@{}[d]|-{\eqref{eq1.2}}&
M\bullet A
\ar[dd]^-{M\bullet (A\circ \eta)}
\ar@{=}[rrr]&&
\ar@{}[dd]|(.6){\eqref{eq:WB}}&
M\bullet A
\ar[dd]_-{M\bullet \Delta}\\
\ar@{}[rddddd]|-{(\RRU)}&
(M\bullet A)\circ (I\bullet I)
\ar[ru]^-\zeta
\ar[d]^-{(M\bullet A)\circ (I\bullet \eta)}\\
&
(M\bullet A)\circ (I\bullet A)
\ar[r]^-\zeta
\ar[d]^-{(M\bullet A)\circ (I\bullet \Delta)}&
M\bullet (A\circ A)
\ar[uurrr]^-{M\bullet \mu}
\ar[r]^(.6){M\bullet (\Delta \circ \Delta)}&
M\bullet ((A\bullet A)\circ (A\bullet A))
\ar[r]^-{M\bullet \zeta}
\ar@{}[d]|-{\eqref{eq1.1}}&
M \bullet (A\circ A)\bullet (A\circ A)
\ar[r]^-{M\bullet \mu \bullet \mu}
\ar[d]^(.3){((M \bullet (A\circ A))\circ \Delta.\eta)\bullet (A\circ A)}&
M \bullet A \bullet A
\ar[d]_(.5){((M \bullet A)\circ \Delta.\eta) \bullet A}\\
&
(M\bullet A)\circ (I\bullet A\bullet A)
\ar[d]^-{(M\bullet A)\circ (((I\bullet A)\circ \Delta.\eta)\bullet A)}
\ar[r]^-{(M\bullet \Delta)\circ (I\bullet A\bullet A)}&
(M\bullet A \bullet A)\circ (I\bullet A\bullet A)
\ar[ru]^-\zeta\ar[r]^-\zeta
\ar[d]^-{(M\bullet A \bullet A)\circ
(((I\bullet A)\circ \Delta.\eta)\bullet A)}&
((M\bullet A)\circ (I\bullet A))\bullet (A\circ A)
\ar[ru]^-{\zeta \bullet (A\circ A)}
\ar[d]^-{((M\bullet A)\circ (I\bullet A)\circ \Delta.\eta)\bullet (A\circ A)}&
((M \bullet (A\circ A))\circ (A\bullet A))\bullet (A\circ A)
\ar[r]^-{\raisebox{6pt}{${}_{
((M \bullet \mu)\circ (A\bullet A))\bullet \mu}$}}\ar@{=}[d]&
((M \bullet A)\circ (A\bullet A))\bullet A
\ar[dd]_-{\zeta\bullet A}\\
&
(M\bullet A)\circ (((I\bullet A)\circ (A\bullet A))\bullet A)
\ar[d]^-{(M\bullet A)\circ (\zeta \bullet A)}&
(M\bullet A \bullet A)\circ (((I\bullet A)\circ (A\bullet A))\bullet A)
\ar[d]^-{(M\bullet A\bullet A)\circ (\zeta \bullet A)}&
((M\bullet A)\circ (I\bullet A)\circ (A\bullet A)) \bullet (A\circ A)
\ar[r]^-{\raisebox{6pt}{${}_{(\zeta\circ (A\bullet A)) \bullet (A\circ A)}$}}
\ar[d]^-{((M\bullet A)\circ \zeta) \bullet (A\circ A)}
\ar@{}[rd]|-{\eqref{eq1.1}}&
((M \bullet (A\circ A))\circ (A\bullet A))\bullet (A\circ A)
\ar[d]^-{\zeta \bullet (A\circ A)}\\
&
(M\bullet A)\circ (A\bullet (A\circ A)\bullet A)
\ar[r]^-{\raisebox{6pt}{${}_{(M\bullet \Delta)\circ (A\bullet (A\circ A)\bullet A)}$}}
\ar[d]^-{(M\bullet A)\circ (A\bullet \mu \bullet A)}&
(M\bullet A \bullet A)\circ (A\bullet (A\circ A)\bullet A)
\ar[r]^-\zeta&
((M\bullet A)\circ (A\bullet (A\circ A)))\bullet (A\circ A)
\ar[r]^-{\zeta \bullet (A\circ A)}&
(M\circ A)\bullet (A\circ A\circ A)\bullet (A\circ A)
\ar[r]^-{\raisebox{6pt}{${}_{(M\circ A)\bullet (\mu\circ A)\bullet \mu}$}}
\ar[d]^-{(M\circ A)\bullet (A\circ \mu)\bullet (A\circ A)}&
(M\circ A)\bullet (A\circ A)\bullet A
\ar[ddd]_-{(M\circ A)\bullet \mu \bullet A}\\
(M\bullet A)\circ A \ar[r]^-{(M\bullet A)\circ \Delta^2}
\ar[d]^-{(M\bullet A)\circ \Delta}&
(M\bullet A)\circ (A \bullet A\bullet A)
\ar[r]^-{(M\bullet \Delta)\circ (A \bullet A\bullet A)}&
(M\bullet A \bullet A)\circ (A \bullet A \bullet A)\ar[r]^-\zeta
\ar@{=}[d]\ar@{}[rrdd]|-{\eqref{eq1.1}}&
((M\bullet A)\circ (A \bullet A))\bullet (A\circ A)
\ar[r]^-{\zeta \bullet (A\circ A)}&
(M\circ A)\bullet (A\circ A)\bullet (A\circ A)\ar@{=}[dd]\\
(M\bullet A)\circ (A \bullet A)
\ar[rr]^-{(M\bullet \Delta)\circ (A \bullet \Delta)}
\ar[d]^-\zeta
&&
(M\bullet A \bullet A)\circ (A \bullet A \bullet A)\ar[d]^-\zeta
\\
(M\circ A)\bullet (A\circ A)\ar[rr]^-{(M\circ A)\bullet (\Delta\circ \Delta)}
\ar[d]^-{(M\circ A)\bullet \mu}
\ar@{}[rrrrd]|-{\eqref{eq:WB}}&&
(M\circ A)\bullet ((A \bullet A)\circ (A \bullet A))
\ar[rr]^-{(M\circ A)\bullet \zeta}&&
(M\circ A)\bullet (A\circ A)\bullet (A\circ A)
\ar[r]^-{(M\circ A)\bullet \mu \bullet \mu}&
(M\circ A)\bullet A \bullet A
\ar[d]_-{(M\circ A)\bullet \varepsilon \bullet A}\\
(M\circ A)\bullet A\ar@{=}[rrrr]&&&&
(M\circ A)\bullet A\ar[ru]^(.33){(M\circ A)\bullet \Delta}\ar@{=}[r]&
(M\circ A)\bullet A}
$$}
\eject

{\white .}\hspace{-3cm} {\em Proof of Lemma \ref{lem:lambda_fork}.}
\vspace{1cm}

\scalebox{1}{
$$\label{page_lambda_fork}
{\white .}\hspace{-2cm}
\xymatrix@C=45pt{
(P\bullet I)\circ L_\bullet \circ A \ar[rrr]^-{\gamma \circ A}
\ar[dddd]^-{(P\bullet I)\circ \alpha}
\ar[rd]^-{(P\bullet I)\circ L_\bullet \circ \Delta}
\ar@{}[rddd]|-{\textrm{Lemma~\ref{lem:mu_colinear}}}&&&
(P\bullet I)\circ A \ar@{=}[r]\ar[d]^-{(P\bullet I)\circ \Delta}
\ar@{}[rddd]|-{\eqref{eq:lambda^0}}&
(P\bullet I)\circ A\ar[ddd]_-{\lambda^0_P}\\
& 
(P\bullet I)\circ L_\bullet \circ (A \bullet A) 
\ar[d]^-{(P\bullet I)\circ\varrho \circ (A \bullet A)}
\ar[rr]^-{\gamma \circ (A \bullet A) }
\ar@{}[rrd]|-{\eqref{eq:gamma}}&&
(P\bullet I)\circ (A\bullet A) \ar@{=}[d]\\
&
(P\bullet I)\circ (L_\bullet \bullet I) \circ (A \bullet A) 
\ar[r]^-{\zeta \circ (A \bullet A)}
\ar[d]^-{(P \bullet I) \circ \zeta}
\ar@{}[rd]|-{\eqref{eq1.1}}&
((P\circ L_\bullet)\bullet I)\circ (A \bullet A) 
\ar[r]^-{(\xi \bullet I)\circ (A \bullet A)} \ar[d]^-\zeta&
(P\bullet I)\circ (A \bullet A) \ar[d]^-\zeta\\
&
(P\bullet I)\circ ((L_\bullet\circ A)\bullet A)\ar[r]^-\zeta
\ar[d]^-{(P\bullet I)\circ (\alpha \bullet A)}&
(P\circ L_\bullet \circ A)\bullet A
\ar[r]^-{(\xi \circ A)\bullet A}
\ar[d]^-{(P\circ \alpha)\bullet A}
\ar@{}[rd]|-{\eqref{eq:pi_L}}&
(P\circ A)\bullet A\ar[r]^-{\pi^R_{P\circ A,A}}
\ar[d]^-{\pi^L_{P,A}\bullet A}&
(P\circ A)\bullet_{R_\circ}A\ar[dd]_-{\pi^L_{P,A}\bullet_{R_\circ} A}\\
(P\bullet I)\circ A\ar@{=}[d]\ar[r]^-{(P\bullet I)\circ \Delta}
\ar@{}[rrd]|-{\eqref{eq:lambda^0}}&
(P\bullet I)\circ (A \bullet A)\ar[r]^-\zeta&
(P\circ A)\bullet A\ar[r]^-{\pi^L_{P,A}\bullet A}
\ar[d]^-{\pi^R_{P\circ A,A}}&
(P\circ_{L_\bullet} A)\bullet A \ar[d]^-{\pi^R_{P\circ_{L_\bullet} A,A}}\\
(P\bullet I)\circ A\ar[rr]_-{\lambda^0_P}&&
(P\circ A)\bullet_{R_\circ} A \ar[r]_-{\pi^L_{P,A}\bullet_{R_\circ} A}&
(P\circ_{L_\bullet} A)\bullet_{R_\circ} A\ar@{=}[r]&
(P\circ_{L_\bullet} A)\bullet_{R_\circ} A}
$$}

\end{landscape}

\end{document}